\documentclass[12pt]{amsart}
\usepackage{amssymb}
\usepackage{amsfonts}
\usepackage{amssymb,latexsym}
\usepackage{enumerate}
\usepackage{mathrsfs}
\usepackage{url}

\makeatletter
\@namedef{subjclassname@2020}{%
  \textup{2020} Mathematics Subject Classification}
\makeatother

  [2002/01/19 v2.2g %
    AMS font definitions%
  ]
\DeclareFontFamily{U}{euf}{}
\DeclareFontShape{U}{euf}{m}{n}{%
  <5><6><7><8><9>gen*eufm%
  <10><10.95><12><14.4><17.28><20.74><24.88>eufm10%
  }{}
\DeclareFontShape{U}{euf}{b}{n}{%
  <5><6><7><8><9>gen*eufb%
  <10><10.95><12><14.4><17.28><20.74><24.88>eufb10%
  }{}

  [2002/01/19 v2.2g %
    AMS font definitions%
  ]
\DeclareFontFamily{U}{msb}{}
\DeclareFontShape{U}{msb}{m}{n}{%
  <5><6><7><8><9>gen*msbm%
  <10><10.95><12><14.4><17.28><20.74><24.88>msbm10%
  }{}

  [2002/01/19 v2.2g %
    AMS font definitions%
  ]
\DeclareFontFamily{U}{msa}{}
\DeclareFontShape{U}{msa}{m}{n}{%
  <5><6><7><8><9>gen*msam%
  <10><10.95><12><14.4><17.28><20.74><24.88>msam10%
  }{}

\newtheorem{theorem}{Theorem}[section]
\newtheorem{lemma}[theorem]{Lemma}

\newtheorem{corollary}[theorem]{Corollary}

\theoremstyle{definition}
\newtheorem{remark}[theorem]{Remark}

\newtheorem{example}[theorem]{Example}

\def\E{\overline E}

\numberwithin{equation}{section} 

\begin{document}

\title[{\tiny Asymptotic expansions for the alternating Hurwitz zeta function}]
{Asymptotic expansions for the alternating Hurwitz zeta function and its derivatives}

\author{Su Hu}
\address{Department of Mathematics, South China University of Technology, Guangzhou 510640, China}
\email{mahusu@scut.edu.cn}

\author{Min-Soo Kim}
\address{Department of Mathematics Education, Kyungnam University, Changwon, Gyeongnam 51767, Republic of Korea}
\email{mskim@kyungnam.ac.kr}


\maketitle

\begin{abstract}
Let
$$
\zeta_E(s,q)=\sum_{n=0}^\infty\frac{(-1)^n}{(n+q)^{s}}
$$
be  the alternating Hurwitz (or Hurwitz-type Euler) zeta function.

In this paper,  we obtain the following asymptotic expansion
of  $\zeta_{E}(s,q)$
$$
\zeta_E(s,q)\sim\frac12 q^{-s}+\frac14sq^{-s-1}-\frac12q^{-s}\sum_{k=1}^\infty\frac{E_{2k+1}(0)}{(2k+1)!}\frac{(s)_{2k+1}}{q^{2k+1}},
$$
as $|q|\to\infty$, where
$E_{2k+1}(0)$ are the special values of odd-order Euler polynomials at 0,
and we also consider representations and bounds for the remainder of the above asymptotic expansion.
In addition, we derive the 
asymptotic expansions for the higher order derivatives of $\zeta_{E}(s,q)$ with respect to its first argument
$$\zeta_{E}^{(m)}(s,q)\equiv\frac{\partial^m}{\partial s^m}\zeta_E(s,q),$$
as $|q|\to\infty$. 
Finally, we also prove a new exact series representation of  
$\zeta_{E}(s,q)$.
\end{abstract}

\def\d{{\rm d}}

\section{Introducrion}\label{intro}

Let $\mathbb N$ be the set of natural numbers, $\mathbb N_0=\mathbb N\cup\{0\},$ 
$\mathbb R$ the field of real numbers and $\mathbb C$ the field of complex numbers.
The Hurwitz zeta function is defined by the following series
\begin{equation}\label{E-2}
\zeta(s,q)=\sum_{n=0}^\infty\frac1{(n+q)^s},
\end{equation}
where Re$(s)>1,~q\neq0,-1,-2,\ldots$
and it can be analytically continued to the entire $s$-plane except for a simple pole at $s=1$ with residue 1
(see \cite[p. 72, Definition 9.6.1]{Cohen} and \cite[p. 88, (1)]{SC}).
Special cases of $\zeta(s,q)$ include the Riemann zeta function
\begin{equation}\label{o-re}
\zeta(s,1)=\zeta(s)=\sum_{n=1}^\infty\frac1{n^s}
\end{equation}
and
$$\zeta\left(s,\frac12\right)=2^s\sum_{n=0}^\infty\frac1{(2n+1)^s}=(2^s-1)\zeta(s).$$
It is known that (see, e.g., \cite[p. 76, Corollary 9.6.10]{Cohen})
\begin{equation}\label{Bernoulli}
	\zeta(1-n,q)=-\frac1n B_n(q)
\end{equation}
for $n\in\mathbb N,$ where $B_n(q)$ denote the $n$th Bernoulli polynomials which
are defined as the coefficient of
$\frac{t^n}{n!}$ in the generating function
\begin{equation}\label{Ber-polynomial}
\frac{te^{qt}}{e^t-1}=\sum_{n=0}^{\infty}B_{n}(q) \frac{t^n}{n!},
\end{equation}
where $|t|<2\pi.$
In particular, $B_{n}(0)=B_{n}$ is the $n$th Bernoulli number.
The Bernoulli numbers and polynomials arise from Bernoulli's calculations of power sums in 1713, that is,
$$
\sum_{j=0}^{m}j^{n}=\frac{B_{n+1}(m+1)-B_{n+1}}{n+1}
$$
for $m,n\in\mathbb N_0$ (see \cite[p. 5, (2.2)]{Sun} and \cite[p. 931, (3.1)]{SK}).

The Hurwitz zeta function $\zeta(s,q)$ has the following asymptotic expansion
\begin{equation}\label{Hurwitz}
\zeta(s,q)\sim\frac12 q^{-s}+\frac{q^{1-s}}{s-1}+q^{1-s}\sum_{k=1}^\infty\frac{B_{2k}}{(2k)!}\frac{(s)_{2k-1}}{q^{2k}},
\end{equation}
as $|q|\to\infty$ in the sector $|\arg q|\leq\pi-\delta,$ with $s$ and $\delta>0$ being fixed, where
$B_{2k}$ are the even-order Bernoulli numbers and
$$(s)_k=s(s+1)\cdots(s+k-1)=\frac{\Gamma(s+k)}{\Gamma(s)}$$
is the Pochhammer symbol (the rising factorial function) (see \cite[p. 25]{Magnus} or \cite[Eq. 25.11.43]{Olver}).
Recently, Nemes \cite{Ne} derived new representations and bounds for the remainder of the asymptotic expansion (\ref{Hurwitz}).
 This work has been applied by Ekera in the theory of quantum computing, that is, by applying one of bounds of Nemes, Ekera proved a lower bound on the probability of Shor's order-finding algorithm successfully recovering the order $r$ in a single run (see
 \cite[Appendix D.2, Claim 6]{Ekera}).
 
It may also be interesting to  consider the derivatives of Hurwitz zeta function $\zeta(s,q)$ at non-positive integer values of $s,$ which have many applications in number theory and mathematical physics
(see \cite{Coffey,EE0,EE1,Ru,Te}). In 1986, by using Watson's Lemma and Laplace's method, Elizalde \cite{EE} gave an asymptotic expansion of the first derivatives
\begin{equation}\label{E-1}
\zeta'(-n,q)\equiv \left. \frac{\partial}{\partial s}\zeta(s,q)\right|_{s=-n}
\end{equation}
for $n\in\mathbb N_0.$ His starting point is Hermite's integral representation for $\zeta(s,q).$
The result has been re-obtined  by Rudaz \cite{Ru} using a direct method connected with the calculus of finite differences.
Earlier in 1984, to study an analogue of the formula of Chowla and Selberg for real quadratic fields, Deninger \cite{De} has investigated the second derivative $\zeta''(0,q)$.
In \cite{EE0} (also see \cite[Chapter 2]{EE1}), Elizalde further obtained some recursive formulas for the higher order derivatives, and presented explicit  formulas for the case of small $m$ and small $s=-n~(n\in\mathbb N)$ 
with respect to $\zeta^{(m)}(s+1,q),$ which completely solve the problem of the calculation of any derivatives of $\zeta(s,q).$
It is also interesting to refer \cite{EE,EE0,EE1,ER,Ru,seri} for many results on the
derivatives $\zeta^{(m)}(s,q)$ at $s=0$ and at $s=-n~(n\in\mathbb{N})$, respectively.

In this paper, we shall consider the alternating Hurwitz (or Hurwitz-type Euler) zeta function
which is defined by the following series
\begin{equation}\label{E-zeta-def}
\zeta_E(s,q)=\sum_{n=0}^\infty\frac{(-1)^n}{(n+q)^{s}},
\end{equation}
where Re$(s)>0,~q\neq0,-1,-2,\ldots$ (see \cite[p. 514, (3.1)]{Choi}, \cite[p. 186, (1.5)]{HKK}, \cite[p. 3, (1.13)]{HK2022},
\cite[p. 1166, (1)]{KMS} and \cite[p. 36, (1.1)]{WZ}).
We see that $\zeta_E(s,q)$ can be analytically
continued to the entire $s$-plane without any pole and it satisfies the following identities:
\begin{equation}\label{J-1}
\zeta_E(s,q+1)+\zeta_E(s,q)=q^{-s},
\end{equation}
\begin{equation}\label{J-2}
\frac{\partial}{\partial q}\zeta_E(s,q)=-s\zeta_E(s+1,q).
\end{equation}
The important special cases include $\zeta_E(s,1):=\zeta_E(s)=\eta(s),$ the Dirichlet eta function 
(see (\ref{Riemann-Euler}))
and $\zeta_E\left(s,\frac12\right)=2^s\beta(s),$  the Dirichlet beta function (see \cite[p. 952]{HK2019} for the definition).
As a function of $q,$ with $s$ fixed, $\zeta_E(s,q)$ is analytic in the domain $|\arg q|<\pi$ and possesses branch-point singularities at non-positive integer values of $q.$
%As a consequence, it then becomes necessary to introduce a series of branch cuts on the negative $q$-axis.
In analogue with (\ref{Bernoulli}), it is known that
(see Corollary \ref{cor1})
\begin{equation}\label{Euler}\zeta_{E}(-n,q)=\frac{1}{2}E_n(q)\end{equation}
for $n\in\mathbb N_0$ (see \cite[p. 520, (3.20)]{Choi}, \cite[p. 188, (2.7)]{HKK} and \cite[p. 41, (3.8)]{WZ}).
Here the Euler polynomials $E_n(q)$ is defined by the generating function
\begin{equation}\label{Euler-polynomial}
\frac{2e^{qt}}{e^t+1}=\sum_{n=0}^{\infty}E_{n}(q) \frac{t^n}{n!},
\end{equation}
where $|t|<\pi$
(see \cite[p. 309, (7)]{Jor}).
Each $E_n(q)$ is a polynomial of degree $n$ with leading coefficient $1.$
The Euler polynomials were introduced by Euler who calculated  the alternating power sums, that is,
\begin{equation*}
\sum_{j=0}^{m}(-1)^{j}j^{n}=\frac{(-1)^{{m}}E_{n}(m+1)+E_{n}(0)}{2}
\end{equation*}
for $m,n\in\mathbb N_0$ (see \cite[p.~804, 23.1.4]{AS}, \cite[p. 5, (2.3)]{Sun} and \cite[p. 932, (3.5)]{SK}).

In algebraic number theory, the alternating Hurwitz zeta function $\zeta_E(s,q)$ represents a partial zeta function of cyclotomic fields in one version of Stark's conjectures (see \cite[p. 4249, (6.13)]{HK-G}),
and its special case, the alternating series (or, equivalently, the Dirichlet eta function),
\begin{equation}~\label{Riemann-Euler}
\zeta_{E}(s,1):=\zeta_{E}(s)=\sum_{n=1}^{\infty}\frac{(-1)^{n-1}}{n^{s}},
\end{equation}
is also a particular case of Witten's zeta functions in mathematical physics (see \cite[p. 248, (3.14)]{Min}).
Indeed, there is a relationship between $\zeta_{E}(s)$ and $\zeta(s)$:
\begin{equation}\label{Riemann-Euler1}
\zeta_E(s)=(1-2^{1-s})\zeta(s).
\end{equation}
In particular, for $s=2n~(n\in\mathbb N),$ it is well-known that
\begin{equation}~\label{posi-val}
\zeta_{E}(2n)=(1-2^{1-2n})\zeta(2n)=(2^{2n-1}-1)\frac{|B_{2n}|}{(2n)!}\pi^{2n}
\end{equation}
(see, e.g., \cite[p. 807, 23.2.16 and 23.2.19]{AS}).

In this paper, inspired by the work of Nemes \cite{Ne}, first we obtain the following asymptotic expansion
of the alternating Hurwitz zeta function $\zeta_{E}(s,q)$
\begin{equation}\label{main-ref}
\zeta_E(s,q)\sim\frac12 q^{-s}+\frac14sq^{-s-1}-\frac12q^{-s}\sum_{k=1}^\infty\frac{E_{2k+1}(0)}{(2k+1)!}\frac{(s)_{2k+1}}{q^{2k+1}},
\end{equation}
as $|q|\to\infty$ in the sector $|\arg q|\leq\pi-\delta,$ with $s$ and $\delta>0$ being fixed, where
$E_{2k+1}(0)$ are the special values of odd-order Euler polynomials at 0 (see Theorem \ref{thm1}),
and we also get representations and bounds for the remainder of the above asymptotic expansion
(see Theorems \ref{asy-pro1}, \ref{asy-thm} and \ref{asy-pro2}).

Then inspired by the works of Elizalde \cite{EE,EE0,EE1}, Nemes \cite{Ne} and Rudaz \cite{Ru}, we derive the 
asymptotic expansions for the higher order derivatives of the alternating Hurwitz zeta function with respect to its first argument
$$\zeta_{E}^{(m)}(s,q)\equiv\frac{\partial^m}{\partial s^m}\zeta_E(s,q),$$
as $|q|\to\infty$ in the sector $|\arg q|\leq\pi-\delta,$ with $s$ and $\delta>0$ being fixed
(see Theorems \ref{thm2} and \ref{thm3} and their corollaries). Finally, we prove a new exact series representation of the alternating Hurwitz zeta function 
$\zeta_{E}(s,q)$ (see Theorem \ref{new-series}).

\section{Preliminaries}
In this section, as a preliminary, we  state the Euler-Boole summation formula. It will serve as a main tool for our approach.
Then as an application, we shall prove an asymptotic series expansion of $\zeta_{E}(s,q)$ (see Theorem \ref{thm1}).
But at first, we need recall the definition of quasi--periodic Euler functions $\E_n(t),$ 
see \cite{BBD} and \cite{HKK2016}  for some details on their properties and applications.

The periodic Euler function $\E_n(t)$ is defined by
 \begin{equation}\label{ae-ft-1}
\E_n(t)=E_n(t) \quad \text{for } 0\leq t<1,
\end{equation}
where $E_n(t)$ denotes the $n$th Euler polynomials (see \eqref{Euler-polynomial}).
We may extend this definition to all $t\in\mathbb{R}$ by setting
\begin{equation}\label{ae-ft}
\E_n(t+1)=-\E_n(t)
\end{equation}
(see \cite[p. 55, (1.3)]{HKK2016}).
In this way, $\E_{n}(t)$ becomes a quasi--periodic function on $\mathbb{R}$, and it has the following Fourier series expansion
\begin{equation}\label{EF-ex}
\E_n(t)=\frac{4n!}{\pi^{n+1}}\sum_{k=0}^\infty\frac{\sin((2k+1)\pi t-\frac12\pi n)}{(2k+1)^{n+1}},
\end{equation}
where $0\leq t\leq1$ if $n\in\mathbb N$ and $0<t<1$ if $n=0$ (see \cite[p.~805, 23.1.16]{AS} and  \cite[p. 57, Lemma 2.1]{HKK2016}).

Using $\E_{n}(t)$ we state the Euler-Boole summation formula. It is obtained by Boole in 1860, but a similar one has been shown by Euler  previously (see  \cite[p. 10, (2.15)]{AC}, \cite[p. 684, Lemma 2]{BBD}, \cite[p. 1205, Theorem 1.2]{CanDa}, \cite[p. 1169, Lemma 2.1]{KMS} and \cite[24.17.1--2]{NIST}).

\begin{lemma}[Euler-Boole summation formula]\label{BSF}
Let $\alpha,\beta$ and $m$ be integers such that $\alpha<\beta$ and $m\in\mathbb N.$ If $f^{(m)}(t)$ is absolutely integrable over $[\alpha,\beta].$
Then
\begin{equation}\label{Boole}
\begin{aligned}
2\sum_{n=\alpha}^{\beta-1}(-1)^nf(n)& = \sum_{k=0}^{m-1}\frac{E_k(0)}{k!}\left((-1)^{\beta-1}f^{(k)}(\beta)+(-1)^{\alpha}f^{(k)}(\alpha) \right) \\
&\quad+\frac1{(m-1)!}\int_{\alpha}^{\beta}\E_{m-1}(-t)f^{(m)}(t)\d t,
\end{aligned}
\end{equation}
where $\E_{n}(t)$ is the $n$th quasi-periodic Euler function.
\end{lemma}

From the above lemma we get the following special case of Euler-Boole's summation formula.  
Similar problems have also been investigated in \cite[pp. 178--179]{Can}, \cite[p. 515, (2.3)]{CP}, \cite[p. 815, Theorem 1.1]{KS} and \cite[p. 17, Theorem 1.4]{Te}.

\begin{lemma}\label{lem1}
Let $f(t)$ be a real- or complex-valued function defined on $0\leq t<\infty.$ 
If the derivative $f^{(2N+2)}(t)$ is absolutely integrable over $[0,1],$ then we have the following expansion 
\begin{equation}\label{Boole-2}
f(0)=\frac12\Delta_+(0)-\frac14\Delta_+'(0)+\frac12\sum_{k=1}^{N-1}\frac{E_{2k+1}(0)}{(2k+1)!}\Delta_+^{(2k+1)}(0)+R_N,
\end{equation}
where $\Delta_+(t):=f(t+1)+f(t)$ and $R_N$ is the remainder given by the formula
\begin{equation}\label{remain}
R_{N}=\frac1{2 (2N+1)!}\int_{0}^{1}\left(E_{2N+1}(0)+\E_{2N+1}(-t)\right)f^{(2N+2)}(t)\d t.
\end{equation}
%(An empty sum, $\sum_{k=1}^{N-1}$ for $N\leq1,$ is interpreted as zero.)
\end{lemma}
\begin{proof}
By setting  $\alpha=0$ and $\beta=1$ in (\ref{Boole}), 
we have
$$f(0) = \frac{1}{2}\sum_{k=0}^{m-1}\frac{E_k(0)}{k!}\left(f^{(k)}(1)+f^{(k)}(0) \right)
+\frac1{2 (m-1)!}\int_{0}^{1}\E_{m-1}(-t)f^{(m)}(t)\d t.$$
Then noticing  that $\Delta_+(t)=f(t+1)+f(t)$ and $E_{0}(0)=1$, $E_{1}(0)=-\frac{1}{2}$, $E_{2k}(0)=0$ for any $k\in\mathbb N$ 
(see, e.g., \cite[p. 5, Corollary 1.1]{Sun}), we get
$$f(0)=\frac12\Delta_+(0)-\frac14\Delta_+'(0)+\frac12\sum_{k=1}^{N-1}\frac{E_{2k+1}(0)}{(2k+1)!}\Delta_+^{(2k+1)}(0)+R_{N},$$
where the remainder term $R_{N}$ is given by
$$R_{N}=\frac1{2 (2N-1)!}\int_{0}^{1}\E_{2N-1}(-t)f^{(2N)}(t)\d t.$$
Thus using 
$$E_{2N-1}(t)=\frac{E_{2N}'(t)}{2N}\quad(N\in\mathbb N)$$
and integration by parts two times, we get
$$\begin{aligned}
R_{N}&=\frac1{2 (2N)!}\int_{0}^{1}\E'_{2N}(-t)f^{(2N)}(t)\d t \\
&=\frac1{2 (2N)!}\left(\left[ \E_{2N}(-t)f^{(2N)}(t)\right]_0^1-\int_0^1 \E_{2N}(-t)f^{(2N+1)}(t)\d t\right) \\
&=-\frac1{2 (2N)!}\int_{0}^{1}\E_{2N}(-t)f^{(2N+1)}(t)\d t \\
&=\frac1{2 (2N+1)!}\int_{0}^{1}\left(E_{2N-1}(0)+\E_{2N+1}(-t)\right)f^{(2N+2)}(t)\d t.
\end{aligned}$$
This completes the proof of Lemma \ref{lem1}.
\end{proof}

\begin{remark}
Using Lemma \ref{lem1}, for an alternating function $f$ which is infinitely differentiable and the derivatives are absolutely integrable,
we have  the asymptotic  expansion (see, for example,  \cite[Chapter VIII]{Hardy})
\begin{equation}\label{E-aym}
f(0)\sim\frac12\Delta_+(0)-\frac14\Delta_+'(0)+\frac12\sum_{k=1}^{\infty}\frac{E_{2k+1}(0)}{(2k+1)!}\Delta_+^{(2k+1)}(0),
\end{equation}
which is in fact Euler's summation formula (\ref{Boole}) without the remainder term. This formula can also be proved in the same way as Rudaz's case \cite[p. 2832, (7)]{Ru}.
It is remarked in \cite[p. 26]{Nor} that this formula has already been found by Euler for polynomials $f$, but without the remainder term.
Also, it should be noted that  our derivation of (\ref{E-aym}) is indeed formal, without considering its convergence
(see \cite[p. 475, (9.71)]{GKP}). Some recent developments for the remainders can be found in the references 
\cite{Lam,Ne} and \cite{seri}.
\end{remark}

\section{Main results} \label{2.2}
In this section, we state our main results, that is, we shall provide asymptotic expansions for the higher order derivatives of the alternating Hurwitz zeta function $\zeta_{E}(s,q)$ in terms of  the lower orders. 
We also present representations and bounds for the remainder of the asymptotic expansion for $\zeta_{E}(s,q)$ stated in Theorem \ref{thm1} below.
All the proofs will be shown in the next section.

First, by applying (\ref{E-aym}) we immediately get the following asymptotic expansion of $\zeta_{E}(s,q),$
as $|q|\to\infty$ in the sector $|\arg q|\leq\pi-\delta,$ with $s$ and $\delta>0$ being fixed (see \cite[p. 2833, (11)]{Ru}). 

\begin{theorem}\label{thm1}
The alternating Hurwitz zeta function has the asymptotic expansion of the form
\begin{equation}\label{main1}
\zeta_E(s,q)\sim\frac12 q^{-s}+\frac14sq^{-s-1}-\frac12q^{-s}\sum_{k=1}^\infty\frac{E_{2k+1}(0)}{(2k+1)!}\frac{(s)_{2k+1}}{q^{2k+1}},
\end{equation}
as $|q|\to\infty$ in the sector $|\arg q|\leq\pi-\delta,$ with $s$ and $\delta>0$ being fixed.
\end{theorem}

\begin{remark}
We refer to Wong's book \cite[pp. 4--42]{Wo} for general concept of asymptotic expansions. Beyond the definition, many properties 
can be found in \cite{Wo}. One may also consult the descriptions in \cite[pp. 20--21]{Er}.
\end{remark}

\begin{remark}
In a recent article \cite{HK2022}, by using the alternating Hurwitz zeta function,
we defined modified Stieltjes constants $\tilde\gamma_k(q)$ from the following Taylor expansion,
\begin{equation}\label{l-s-con}
\zeta_E(s,q)=\sum_{k=0}^\infty\frac{(-1)^k\tilde\gamma_k(q)}{k!}(s-1)^k.
\end{equation}
Then by using the asymptotic expansion 
(\ref{main1}), we  obtained an  asymptotic expansion of $\tilde\gamma_k(q).$ See \cite[Theorem 3.17]{HK2022} and its proof.
\end{remark}

For $|\arg q|\leq\pi$ and any positive integer $N\in\mathbb N,$ we define the $N$th remainder term $R_{N}(s,q)$ 
of the asymptotic expansion \eqref{main1} by the equality
\begin{equation}\label{rem-te1}
\zeta_E(s,q)=\frac12 q^{-s}-\frac{1}2q^{-s}\left(\sum_{k=0}^{N-1}\frac{E_{2k+1}(0)}{(2k+1)!}\frac{(s)_{2k+1}}{q^{2k+1}}
+R_{N}(s,q)\right)
\end{equation}
(see \cite[p. 2, (1.3)]{Ne}).

Equation (\ref{rem-te1}) implies the following known result on the special values of $\zeta_E(s,q)$ 
at $s=-n$ for $n\in\mathbb N_0$ (see \cite[p. 520, (3.20)]{Choi}, \cite[p. 188, (2.7)]{HKK} and \cite[p. 41, (3.8)]{WZ}).
An elementary proof will be given in the next section.

\begin{corollary}\label{cor1}
The values of $\zeta_E(s,q)$ at non-positive integers are expressed by:
$$\zeta_E(-n,q)=\frac12 E_n(q),\quad n\in\mathbb N_0,$$
where $E_n(q)$ denotes the $n$th Euler polynomials.
In particular, for $n=0,$ we obtain $\zeta_E(0,q)=\frac12.$
\end{corollary}

In the following, we shall consider new representations for the remainder term $R_{N}(s,q)$ in (\ref{rem-te1}).

\begin{theorem}\label{asy-pro1}
Let $N\in\mathbb N$ and {\rm Re}$(s)>-2N-2.$ Then we have
$$
R_{N}(s,q)=\frac{(s)_{2N+2}}{(2N+1)!}q^s\int_0^\infty\frac{E_{2N+1}(0)-\E_{2N+1}(-t)}{(q+t)^{s+2N+2}}{\d t}
$$
for $|\arg q|<\pi,$ where $\E_n(t)$ denotes the quasi--periodic Euler function.
Moreover, 
$$
\begin{aligned}
\frac{E_{2N+3}(0)}{(2N+3)!}\frac{(s)_{2N+3}}{q^{2N+3}}+\frac{E_{2N+1}(0)}{(2N+1)!}\frac{(s)_{2N+1}}{q^{2N+1}} 
< R_{N}(s,q)
<\frac{E_{2N+1}(0)}{(2N+1)!}\frac{(s)_{2N+1}}{q^{2N+1}}
\end{aligned}
$$
for $|\arg q|<\frac\pi2$ and {\rm Re}$(s)>-2N,$ where $N=2m-1~(m\in\mathbb N).$
\end{theorem}

We will also prove that the following expansion (see \cite[p. 3, Theorem 1.3]{Ne}).

\begin{theorem}\label{asy-thm}
If $q$ is a positive real number and $s$ is an arbitrary real number $\mathbb R$ such that $s\neq1$ and 
$s>-2N-1~(N\in\mathbb N),$ then
$$
R_{N}(s,q)=(-1)^{N+1}4\frac{(s)_{2N+1}}{\pi^{2N+2}q^{2N+1}}
\sum_{k=0}^\infty\frac{{\it\Pi}_{s+2N+1}((2k+1)\pi q)}{(2k+1)^{2N+2}}
$$
and
$$
R_{N}(s,q)=\frac{E_{2N+1}(0)}{(2N+1)!}\frac{(s)_{2N+1}}{q^{2N+1}}\theta_N(s,q),
$$
where $0<\theta_N(s,q)<1$ is a suitable number that depends on $s,q$ and $N.$
\end{theorem}

The following theorem is an analogue of \cite[p. 3, Theorem 1.2]{Ne} by Nemes.
\begin{theorem}\label{asy-pro2}
Let $N\in\mathbb N$ and $s$ be an arbitrary complex number $\mathbb C$ such that {\rm Re}$(s)>-2N-1.$ We have
\begin{equation}\label{pro2-1}
|R_{N}(s,q)|\leq
\frac{|E_{2N+1}(0)|}{(2N+1)!}\frac{|(s)_{2N+1}|}{|q|^{2N+1}}\sup_{r\geq1}|{\it \Pi}_{s+2N+1}((2r+1)\pi q)|
\end{equation}
and
\begin{equation}\label{pro2-2}
\begin{aligned}
|R_{N}(s,q)|&\leq
\frac{|E_{2N+1}(0)|}{(2N+1)!}\frac{|(s)_{2N+1}|}{|q|^{2N+1}}\frac{|s+2N+1|}{{\rm Re}(s)+2N+1} \\
&\quad\times
\sec^{{\rm Re}(s)+2N+2}\left(\frac{\arg q}{2}\right)\max\left(1,e^{-{\rm Im}(s)\arg q}\right)
\end{aligned}
\end{equation}
provided that $|\arg q|<\pi,$ where ${\it \Pi}_p(w)$ denotes  \cite[p. 2]{Ne}
\begin{equation}\label{Pi}
{\it \Pi}_p(w)=\frac{w^p}2\left(e^{\frac\pi2 {\rm i}p}e^{{\rm i}w}\Gamma(1-p,we^{\frac\pi2 {\rm i}})
+e^{-\frac\pi2 {\rm i}p}e^{-{\rm i}w}\Gamma(1-p,we^{-\frac\pi2 {\rm i}})\right).
\end{equation}
Here $\Gamma(1-p,w)$ is the incomplete gamma function.
\end{theorem}

\begin{remark}
It needs to mention that since the radius of convergence of the series in (\ref{Euler-polynomial}) equals to $\pi,$ we have a rough estimate
\begin{equation}\label{ro-est}
\frac{E_{2k+1}(0)}{(2k+1)!}=O\left(\pi^{-(2k+1)}\right)
\end{equation}
as $k\to\infty.$ This estimate can be refined by using (\ref{EF-ex}). Since the series  (\ref{EF-ex}) take values between 0 and 1,
for every $k\in\mathbb N,$ we have
\begin{equation}\label{ro-est2}
\frac1{\pi^{2k+1}} <(-1)^{k+1}\frac{E_{2k+1}(0)}{(2k+1)!}<\frac2{\pi^{2k+1}}.
\end{equation}
We can evaluate the following identity (see \cite[p. 954]{HK2019})
\begin{equation}\label{re1}
(-1)^{k+1}\frac{\pi^{2k+2}E_{2k+1}(0)}{4(2k+1)!}=\sum_{n=0}^\infty\frac1{(2n+1)^{2k+2}}=\left(1-2^{-2k-2}\right)\zeta(2k+2)
\end{equation}
for $k\in\mathbb N.$ Therefore, for sufficiently large $k\in\mathbb N,$ we obtain the approximation
\begin{equation}\label{re2}
(-1)^{k+1}\frac{\pi^{2k+2}E_{2k+1}(0)}{4(2k+1)!}\approx\zeta(2k+2).
\end{equation}
\end{remark}

\begin{theorem}\label{thm2}
The derivative of the alternating Hurwitz zeta function has the following asymptotic expansion 
$$\zeta_E'(s,q)\sim\frac14 q^{-s-1}-\zeta_E(s,q)\log q-\frac12\sum_{k=1}^\infty E_{2k+1}(0)
\sum_{j=0}^{2k}\frac{(s)_j}{j!(2k-j+1)}q^{-s-2k-1},$$
as $|q|\to\infty$ in the sector $|\arg q|\leq\pi-\delta,$ with $s$ and $\delta>0$ being fixed.
\end{theorem}

\begin{corollary}\label{n=01} 
For each $n\in\mathbb N_0$ the following asymptotic expansion holds
$$
\begin{aligned}
\zeta_E'(-n,q)\sim\frac14 q^{n-1}&-\frac12E_n(q)\log q \\
&-\frac12\sum_{k=1}^\infty E_{2k+1}(0)\sum_{j=0}^{\min(n,2k)}\binom nj\frac{(-1)^j}{2k-j+1}q^{n-2k-1},
\end{aligned}
$$
as $|q|\to\infty$ in the sector $|\arg q|\leq\pi-\delta,$ with $\delta>0$ being fixed. 
\end{corollary}

\begin{example}\label{ex-n=01}
In particular, setting $n=0$ and 1  respectively in Corollary \ref{n=01},
we obtain the expansions
$$\begin{aligned}
&\zeta_E'(0,q)\sim-\frac12\log q +\frac14q^{-1}-\frac12\sum_{k=1}^\infty\frac{E_{2k+1}(0)}{2k+1}q^{-2k-1}, \\
&\zeta_E'(-1,q)\sim-\frac12q\log q+\left(\frac14+\frac14\log q\right)+\frac12\sum_{k=1}^\infty\frac{E_{2k+1}(0)}{(2k+1)(2k)}q^{-2k},
\end{aligned}$$
as $|q|\to\infty$ in the sector $|\arg q|\leq\pi-\delta,$ with $\delta>0$ being fixed.
\end{example}

\begin{corollary}\label{cor2}
For each $n\geq2$ the following asymptotic expansion holds
$$\begin{aligned}
\zeta_E'(-n,q)&\sim\frac14 q^{n-1}-\frac12E_n(q)\log q \\
&\quad-\frac12\sum_{k=1}^{\left\lfloor \frac{n}{2}\right\rfloor} E_{2k+1}(0)\sum_{j=0}^{2k}\binom nj\frac{(-1)^j}{2k-j+1}q^{n-2k-1} \\
&\quad+(-1)^{n+1}\frac{n!}2\sum_{k=\left\lfloor \frac{n}{2}\right\rfloor+1}^\infty\frac{E_{2k+1}(0)}{(2k+1)(2k)\cdots(2k-n+1)}q^{n-2k-1},
\end{aligned}$$
as $|q|\to\infty$ in the sector $|\arg q|\leq\pi-\delta,$ with $\delta>0$ being fixed.
\end{corollary}

\begin{example}
As examples, letting $n=2,3,4$ and $5$ in Corollary \ref{cor2}, we get the following formulas:
$$
\zeta_E'(-2,q)\sim-\left(\frac12\log q\right)q^2+\left(\frac14+\frac12\log q\right)q
-\sum_{k=1}^\infty\frac{E_{2k+1}(0)}{(2k+1)(2k)(2k-1)}q^{-2k+1},
$$
$$
\begin{aligned}
\zeta_E'(-3,q)&\sim-\left(\frac12\log q\right)q^3+\left(\frac14+\frac34\log q\right)q^2-\left(\frac{11}{48}+\frac18\log q\right)
 \\
&\quad+3\sum_{k=2}^\infty\frac{E_{2k+1}(0)}{(2k+1)(2k)(2k-1)(2k-2)}q^{-2k+2},
\end{aligned}
$$
$$
\begin{aligned}
\zeta_E'(-4,q)&\sim-\left(\frac12\log q\right)q^4+\left(\frac14+\log q\right)q^3-\left(\frac{13}{24}+\frac12\log q\right)q
+\frac1{20}q^{-1}
 \\
&\quad-12\sum_{k=3}^\infty\frac{E_{2k+1}(0)}{(2k+1)(2k)(2k-1)(2k-2)(2k-3)}q^{-2k+3},
\end{aligned}
$$
$$
\begin{aligned}
\zeta_E'(-5,q)&\sim-\left(\frac12\log q\right)q^5+\left(\frac14+\frac54\log q\right)q^4-\left(\frac{47}{48}+\frac54\log q\right)q^2 \\
&\quad+\left(\frac{137}{240}+\frac14\log q\right) +\frac{17}{672}q^{-2}\\
&\quad+60\sum_{k=4}^\infty\frac{E_{2k+1}(0)}{(2k+1)(2k)(2k-1)(2k-2)(2k-3)(2k-4)}q^{-2k+4},
\end{aligned}
$$
as $|q|\to\infty$ in the sector $|\arg q|\leq\pi-\delta,$ with $\delta>0$ being fixed.
\end{example}

Denote by $\log^jq=(\log q)^j$ for $j\geq1.$ In the following, using Theorem \ref{thm1}, we may represent $\zeta_E^{(m)}(s,q),$ the higher order derivatives of the alternating Hurwitz zeta function, in terms of the lower orders (also see \cite[(11) and (13)]{EE0}).
As conventions empty sums will be taken to be zero.

\begin{theorem}\label{thm3}
For each $m\geq2$ the function $\zeta_E^{(m)}(s,q)$ has the asymptotic expansion corresponding to the derivatives of 
lower orders:
$$\begin{aligned}
\zeta_E^{(m)}(s,q)
&\sim\left[\left(\frac{\partial}{\partial s}+\log q\right)^m-\left(\frac{\partial}{\partial s}\right)^m\right]\zeta_E(s,q)
-\frac12\Sigma_m(s,q) \\
&=-\sum_{j=1}^m\binom mj\zeta_E^{(m-j)}(s,q)\log^j q
-\frac12\Sigma_m(s,q), \\
\end{aligned}$$
as $|q|\to\infty$ in the sector $|\arg q|\leq\pi-\delta,$ with $s$ and $\delta>0$ being fixed, where
\begin{equation}\label{ckm}
\Sigma_m(s,q)=\sum_{k=1}^\infty q^{-s-2k-1}c_{k}^m(s)
\end{equation}
and
$$
c_{k}^m(s)= E_{2k+1}(0)\sum_{j_1=0}^{2k}\frac1{2k-j_1+1}\sum_{j_2=0}^{j_1-1}\frac1{j_1-j_2}
\cdots\sum_{j_m=0}^{j_{m-1}-1}\frac{(s)_{j_m}}{j_m!(j_{m-1}-j_m)}.
$$
\end{theorem}

If we restrict ourselves to the particular values $s=-n~(n\in\mathbb N_0),$ then we obtain the following result
(see \cite[(14)]{EE0}).

\begin{corollary}\label{cor3}
For each $m\geq2$ the function $\zeta_E^{(m)}(s,q)$ has the asymptotic expansion by evaluating at nonpositive integer values of 
$s$ :
$$\zeta_E^{(m)}(-n,q)\sim-\sum_{j=1}^m\binom mj\zeta_E^{(m-j)}(-n,q)\log^j q-\frac12\Sigma_m(-n,q),$$
as $|q|\to\infty$ in the sector $|\arg q|\leq\pi-\delta,$ with $\delta>0$ being fixed, where $n\in\mathbb N_0,$ 
$$\Sigma_m(-n,q)=\sum_{k=1}^\infty q^{n-2k-1}c_{k}^m(-n)$$
and
$$c_{k}^m(-n)=  E_{2k+1}(0)\sum_{j_1=0}^{2k}\frac1{2k-j_1+1}\sum_{j_2=0}^{j_1-1}\frac1{j_1-j_2}
\cdots\sum_{j_m=0}^{\mu_m}\binom{n}{j_m}\frac{(-1)^{j_m}}{j_{m-1}-j_m}
$$
in which $\mu_m=\min(n,j_{m-1}-1).$
\end{corollary}

\begin{remark}
We remark here that the formula in Theorem \ref{thm2} for $\zeta_E'(s,q)$ cannot be obtained from Theorem \ref{thm3} directly by setting $m=1$ since
it contains an additional term in $q^{-s-1}.$
\end{remark}

\begin{remark}
The following result was shown by Williams and Zhang in \cite[p. 40, Proposition 3]{WZ}:
\begin{equation}\label{eu-zeta-p}
\zeta_E'(0,q)=\log\frac{\Gamma\left(\frac q2\right)}{\Gamma\left(\frac{1+q}2\right)}-\frac12\log 2.
\end{equation}
Thus by Example \ref{ex-n=01} and (\ref{eu-zeta-p}), and replacing $q$ by $2q,$ we have the following asymptotic expansion:
\begin{equation}\label{eu-type-asym}
\log\frac{\Gamma(q)}{\Gamma\left(q+\frac{1}2\right)}
\sim-\frac12\log q -\sum_{k=0}^\infty\frac{E_{2k+1}(0)}{2^{2k+2}(2k+1)}q^{-2k-1},
\end{equation}
as $|q|\to\infty$ in the sector $|\arg q|<\pi.$ We note that (\ref{eu-type-asym}) can also be deduced from
the following identity
\begin{equation}\label{b-e-eq1}
B_{n+1}\left(\frac12\right)-B_{n+1}=\frac{n+1}{2^{n+1}}E_n(0)
\end{equation}
for $n\in\mathbb N_0$ (see \cite[p. 806, 23.1.27]{AS}) and the following formula
\begin{equation}\label{c-gam-eq}
\log\Gamma\left(q+t\right)\sim\left(q+t-\frac12\right)\log q-q+\frac12\log(2\pi)
+\sum_{n=1}^\infty\frac{(-1)^{n+1}B_{n+1}(t)}{n(n+1)}q^{-n}.
\end{equation}
The above formula can be found in \cite[p. 32, (4)]{Lu} and \cite[p. 1005, (2.10)]{NE}.

Using (\ref{eu-type-asym}), (\ref{b-e-eq1}) and $E_{2k}(0)=0~(k\in\mathbb N),$ 
and applying (1.5) and Theorem 2.1 (with $t=\frac12$ and $s=0$) of \cite{BE} (also see \cite[p. 356, (1.5)]{BE}),
we get the following asymptotic expansion 
\begin{equation}\label{eu-type-asym1}
\begin{aligned}
\frac{\Gamma\left(q+\frac{1}2\right)}{\Gamma(q)}
&\sim \sqrt q \exp\left(\frac12\sum_{n=1}^\infty\frac{(-1)^{n+1}E_{n}(0)}{2^{n}n}q^{-n}\right) \\
&\sim  \sqrt q \exp\left(\sum_{n=1}^\infty\frac{(-1)^{n+1}}{n(n+1)}\left[B_{n+1}\left(\frac12\right)-B_{n+1}\right]q^{-n}\right) \\
&\sim \sqrt q \left(1-\frac{1}{8q}+\frac{1}{128q^2}+\frac{5}{1024q^3}-\frac{21}{32768q^4}+\cdots\right),
\end{aligned}
\end{equation}
as $|q|\to\infty$ in the sector $|\arg q|<\pi.$ 
This leads to the approximation
$$\frac{\Gamma\left(q+\frac{1}2\right)}{\Gamma(q)}\approx \sqrt q.$$ For (\ref{eu-type-asym1}), the following asymptotic formula is very well-known and useful (see, e.g., \cite[p. 7, (37)]{SC}):
$$\frac{\Gamma(q+\alpha)}{\Gamma(q+\beta)}=q^{\alpha-\beta}\left[ 1+\frac{(\alpha-\beta)(\alpha+\beta-1)}{2q} +O(q^{-2})\right]$$ 
$$(|q|\to\infty; ~|\arg(q)|\leq\pi-\epsilon;~|\arg(q+\alpha)|\leq\pi-\epsilon;~0<\epsilon<\pi),$$
where $\alpha$ and $\beta$ are bounded complex numbers.
We refer to \cite{BE} for detailed analysis of asymptotic expansions 
for the ratio of two gamma functions.
\end{remark}

We close this section by introducing a new exact series representation of the convergent expansion of $\zeta_E(s,q).$ 

\begin{theorem}\label{new-series}
If $\{N_k\}_{k=0}^\infty$ is an arbitrary sequence of positive integers and $s$ is any complex number $\mathbb C$ such that 
{\rm Re}$(s)>-2N_k$ for each $k\in\mathbb N_0,$ then the alternating Hurwitz zeta function has the exact expansion
$$
\begin{aligned}
\zeta_E(s,q)&=\frac12{q^{-s}}+\frac14sq^{-s-1}-2q^{-s}\left( \sum_{k=0}^\infty\sum_{n=1}^{N_k-1}(-1)^{n-1}
\frac{(s)_{2n+1}}{((2k+1)\pi)^{2n+2}q^{2n+1}} \right. \\
&\quad+\left. \sum_{k=0}^\infty(-1)^{N_k-1}\frac{(s)_{2N_k+1}{\it\Pi}_{s+2N_k+1}\left({(2k+1)\pi q}\right)}{((2k+1)\pi)^{2N_k+2}q^{2N_k+1}} \right)
\end{aligned}
$$
provided that $|\arg q|<\pi,$ where ${\it \Pi}_p(w)$ is defined by equation (\ref{Pi}).
\end{theorem}

Notice that the above representation was originally established 
by Nemes \cite[p. 3, (1.11)]{Ne} in a slightly different form for the Hurwitz zeta function $\zeta(s,q)$.
 
According to Abramowitz and Stegun's handbook \cite[pp.~807--808]{AS}, the function
\begin{equation}\label{lam-def}
\lambda(s)=\sum_{k=0}^\infty\frac1{(2k+1)^s}=(1-2^{-s})\zeta(s),\quad\text{Re}(s)>1,
\end{equation}
is usually named the Dirichlet lambda function.
In addition, the special values of $\lambda(s)$ at any even positive integer $2n\;(n\in\mathbb N)$
are calculated by the following formula:
\begin{equation}\label{lam-ft-eq}
\lambda(2n)=(-1)^{n}\frac{\pi^{2n}}{4(2n-1)!}E_{2n-1}(0)
\end{equation}
(see \cite[p. 954, (1.10)]{HK2019}). If we put $N_k=N\in\mathbb N\;(k\in\mathbb N_0)$ in Theorem \ref{new-series} and
using (\ref{lam-ft-eq}), we have the following Poincar\'e type expansion.

\begin{corollary}\label{new-series-co}
If $s$ is any complex number such that 
{\rm Re}$(s)>-2N,$ then the alternating Hurwitz zeta function has the representation
$$
\begin{aligned}
\zeta_E(s,q)&=\frac12{q^{-s}}+\frac14sq^{-s-1}-\frac12q^{-s}\sum_{n=1}^{N-1}
\frac{E_{2n+1}(0)}{(2n+1)!}\frac{(s)_{2n+1}}{q^{2n+1}} +\mathcal R_N(s,q),
\end{aligned}
$$
provided that $|\arg q|<\pi,$ where
$$\mathcal R_N(s,q)= 2q^{-s}\sum_{k=0}^\infty(-1)^{N}\frac{(s)_{2N+1}{\it\Pi}_{s+2N+1}\left({(2k+1)\pi q}\right)}
{((2k+1)\pi)^{2N+2}q^{2N+1}}$$
and ${\it \Pi}_p(w)$ is defined by equation (\ref{Pi}). Moreover, the remainder term $\mathcal R_N(s,q)$ satisfies the
estimate 
$$\mathcal R_N(s,q)=O(q^{-s-2N-1})$$ 
as $q\to\infty,$ and can be expressed as $\mathcal R_N(s,q)=-\frac12q^{-s}R_N(s,q).$
\end{corollary}

\section{Proofs of the  results}\label{proofs}

In this section, we prove the  results stated  in Section \ref{2.2}.

\subsection*{Proof of Theorem \ref{thm1}.}
We prove this theorem by using Lemma \ref{lem1} and \eqref{E-aym}.

For $|q|\to\infty$, in the sector $|\arg q|\leq\pi-\delta$ with $\delta>0$ being fixed,
we will derive an asymptotic expansion for $\zeta_E(s,q).$  
For fixed $s,q\in\mathbb C$ and for $t\in\mathbb R,$ considering the function (see \cite[p. 2833]{Ru})
$$f(t)=\zeta_E(s,q+t).$$
It satisfies all the assumption of Lemma \ref{lem1}. Then from the difference equation (\ref{J-1}) and the definition of $\Delta_+(t)=f(t+1)+f(t),$
we have
\begin{equation}\label{delta0}
\Delta_+(0)=f(1)+f(0)=\zeta_E(s,q+1)+\zeta_E(s,q)=q^{-s}
\end{equation}
and
\begin{equation}\label{delta}
\begin{aligned}
\Delta_+'(0)&=\frac{\d}{\d t}\left(f(t+1)+f(t)\right)\biggl|_{t=0} \\
&=\frac{\d}{\d t}\left(\zeta_E(s,q+t+1)+\zeta_E(s,q+t)\right)\biggl|_{t=0} \\
&=\frac{\d}{\d t}\left((q+t)^{-s}\right)\biggl|_{t=0} \\
&=-sq^{-s-1}.
\end{aligned}
\end{equation}
Repeating this procedure $k$ times, we immediately get
\begin{equation}\label{delta1}
\Delta_+^{(k)}(0)=(-1)^k(s)_kq^{-s-k}
\end{equation}
for $k\in\mathbb N_0.$
Substituting (\ref{delta0}), (\ref{delta}) and (\ref{delta1}) into Lemma  \ref{lem1} and using \eqref{E-aym}, the result follows.

\subsection*{Proof of Corollary \ref{cor1}.}
For $N>n$, by putting $s=-n~(n\in\mathbb N_0)$ in (\ref{rem-te1}), we have
$$\begin{aligned}
\zeta_E(-n,q)&=\frac12q^n-\frac12\sum_{k=1}^N\frac{E_k(0)}{k!}\frac{(-n)_k}{q^{k-n}}
-\frac12\frac{E_{N+1}(0)}{(N+1)!}\frac{(-n)_{N+1}}{q^{N-n+1}} -\cdots \\
&\quad(\text{by using $E_{k}(0)=0$ if any even $k\geq2$}) \\
&=\frac12q^n+\frac12\sum_{k=1}^N(-1)^{k+1}\binom nkE_k(0)q^{n-k} +\cdots \\
&\quad(\text{here the series terminates of itself}) \\
&=\frac12\sum_{k=0}^n\binom nkE_k(0)q^{n-k} \\
&=\frac12E_n(q),
\end{aligned}$$
where we have used $E_0(0)=1$ and the last step follows from the fact that
$E_n(q)=\sum_{k=0}^n\binom nkE_k(0)q^{n-k}$
(see \cite[p. 804, 23.1.7]{AS} and \cite[p. 3, (1.4)]{Sun}).
This is our result.

\subsection*{Proof of Theorem \ref{asy-pro1}.}
It is easy to show that the remainder term $R_{N}(s,q)$ can be expressed in the following form
\begin{equation}\label{rem}
R_{N}(s,q)=\frac{q^s}{\Gamma(s)}\int_0^\infty\left(\frac{2}{e^t+1}-1-\sum_{k=0}^{N-1}\frac{E_{2k+1}(0)}{(2k+1)!}
t^{2k+1}\right)t^se^{-qt}\frac{\d t}t
\end{equation}
provide that $|\arg q|<\frac\pi2$ and Re$(s)>-2N$ (see \cite[p. 4, (2.1)]{Ne}).
Using $E_1(0)=-\frac12$ and $E_{2k+1}(0)=-E_{2k+1}(1),$ 
it follows from \cite[(2.5)]{CP} that 
\begin{equation}\label{cp-1}
\sum_{k=0}^{2m}\frac{E_{2k+1}(0)}{(2k+1)!}t^{2k+1}<\frac{2}{e^t+1}-1<\sum_{k=0}^{2m-1}\frac{E_{2k+1}(0)}{(2k+1)!}t^{2k+1},
\end{equation}
for $t>0$ and $m\in\mathbb N.$ 
Since $E_{4m-1}(0)>0$ and $E_{4m+1}(0)<0$ for all $m\in\mathbb N,$
(\ref{cp-1}) can be written in the form
\begin{equation}\label{cp-2}
\begin{aligned}
\frac{E_{4m+1}(0)}{(4m+1)!}t^{4m+1}+\frac{E_{4m-1}(0)}{(4m-1)!}t^{4m-1}
&<\frac{2}{e^t+1}-1-\sum_{k=0}^{2m-2}\frac{E_{2k+1}(0)}{(2k+1)!} t^{2k+1}\\
&<\frac{E_{4m-1}(0)}{(4m-1)!}t^{4m-1}
\end{aligned}
\end{equation}
for $t>0$ and $m\in\mathbb N.$
Because
$$q^{-(s+\ell)}=\frac1{\Gamma(s+\ell)}\int_0^\infty t^{s+\ell} e^{-qt}\frac{\d t}{t}\quad (\text{Re}(q)>0;\,s+\ell\not\in-\mathbb N),$$
by (\ref{rem}) and  (\ref{cp-2}) we have
\begin{equation}\label{cp-3}
\begin{aligned}
\frac{E_{4m+1}(0)}{(4m+1)!}\frac{\Gamma(s+4m+1)}{\Gamma(s)}\frac1{q^{4m+1}}&+\frac{E_{4m-1}(0)}{(4m-1)!}\frac{\Gamma(s+4m-1)}{\Gamma(s)}\frac1{q^{4m-1}} \\
&< R_{2m-1}(s,q)\\
&<\frac{E_{4m-1}(0)}{(4m-1)!}\frac{\Gamma(s+4m-1)}{\Gamma(s)}\frac1{q^{4m-1}}.
\end{aligned}
\end{equation}
Setting $N=2m-1~(m\in\mathbb N)$ in (\ref{cp-3}), we obtain the following estimation for the remainder term $R_{N}(s,q)$
\begin{equation}\label{cp-4}
\begin{aligned}
\frac{E_{2N+3}(0)}{(2N+3)!}\frac{(s)_{2N+3}}{q^{2N+3}}+\frac{E_{2N+1}(0)}{(2N+1)!}\frac{(s)_{2N+1}}{q^{2N+1}} 
< R_{N}(s,q)
<\frac{E_{2N+1}(0)}{(2N+1)!}\frac{(s)_{2N+1}}{q^{2N+1}} .
\end{aligned}
\end{equation}
In what follows, we apply Lemma \ref{BSF} to study  the remainder term $R_{N}(s,q)$ by letting $\alpha=0,\beta\to\infty$ and $f(t)=(t+q)^{-s}$ with Re$(s)>0$.
The $k$th derivative of $f(t)$ is expressed by
$$f^{(k)}(t)=(-1)^k(s)_k(t+q)^{-(s+k)},$$
thus $f^{(k)}(\beta)\to0$ as $\beta\to\infty.$ Replacing these ingredients by (\ref{Boole}), we obtain the following 
formula
\begin{equation}\label{a-boole}
\begin{aligned}
\zeta_E(s,q)&=\frac12\sum_{k=0}^{m-1}(-1)^k\frac{E_k(0)}{k!}\frac{(s)_k}{q^{s+k}}
+\frac{(-1)^m(s)_m}{2(m-1)!}\int_0^\infty\frac{\E_{m-1}(-t)}{(q+t)^{s+m}}\d t \\
&=\frac12 q^{-s}-\frac12q^{-s}\sum_{k=0}^{N-1}\frac{E_{2k+1}(0)}{(2k+1)!}\frac{(s)_{2k+1}}{q^{2k+1}} \\
&\qquad\qquad\qquad\qquad\qquad
+\frac{(-1)^{2N}(s)_{2N}}{2(2N-1)!}\int_0^\infty\frac{\E_{2N-1}(-t)}{(q+t)^{s+2N}}\d t.
\end{aligned}
\end{equation}
Since $E_{2k}(0)=0$ for $k\in\mathbb N,$ comparing (\ref{rem-te1}) and (\ref{a-boole}),
it is easily seen that $R_{N}(s,q)$ can be written in the integral representation
\begin{equation}\label{rem2}
R_{N}(s,q)=-\frac{(s)_{2N}}{(2N-1)!}q^s\int_0^\infty\frac{\E_{2N-1}(-t)}{(q+t)^{s+2N}}{\d t},
\end{equation}
which is valid for $|q|>0$ and Re$(s)>-2N.$ Integrating by parts two times and by using the property \cite[(2.2)]{HKK2016}
$$E_N(t)=\frac{E'_{N+1}(t)}{N+1}\quad(N\in\mathbb N_0),$$
we have
\begin{equation}\label{parts}
\begin{aligned}
R_{N}(s,q)&=-\frac{(s)_{2N}}{(2N-1)!}q^s\left(\frac1{2N}\frac{E_{2N}(0)}{q^{s+2N}}
-\frac{s+2N}{2N}\int_0^\infty\frac{\E_{2N}(-t)}{(q+t)^{s+2N+1}}{\d t}\right) \\
&=\frac{(s)_{2N+1}}{(2N)!}q^s
\int_0^\infty\frac{\E_{2N}(-t)}{(q+t)^{s+2N+1}}{\d t} \\
&=\frac{E_{2N+1}(0)}{(2N+1)!}\frac{(s)_{2N+1}}{q^{2N+1}} - \frac{(s)_{2N+2}}{(2N+1)!}q^s
\int_0^\infty\frac{\E_{2N+1}(-t)}{(q+t)^{s+2N+2}}{\d t} \\
&=\frac{(s)_{2N+2}}{(2N+1)!}q^s\int_0^\infty\frac{E_{2N+1}(0)-\E_{2N+1}(-t)}{(q+t)^{s+2N+2}}{\d t},
\end{aligned}
\end{equation}
since $E_{2N}(0)=0$ for $N\in\mathbb N.$
The analytic continuation procedure then allows one to conclude that \eqref{parts} is valid under the more general conditions
$|\arg q|<\pi$ and Re$(s)>-2N-2$ (see \cite[(2.6) and (2.7)]{Ne} and \cite[\S64]{Kn}).
This competes the proof.

\subsection*{Proof of Theorem \ref{asy-thm}.}
For the proof, first we note that (see \cite[(2.4)]{CP1})
\begin{equation}\label{E-B}
E_{2k-1}(0)=\frac{2(1-2^{2k})}{2k}B_{2k}\quad\text{and}\quad E_{2k-1}(1)=-E_{2k-1}(0)
\end{equation}
for $k\in\mathbb N,$ where $B_k$ is the Bernoulli numbers.
Then for $t>0$ and $N\in\mathbb N,$ we have (see \cite[Theorem 2.1]{CP1})
\begin{equation}\label{asy-thm-pf1}
\frac{2}{e^t+1}-1-\sum_{k=0}^{N-1}\frac{E_{2k+1}(0)}{(2k+1)!}t^{2k+1}=(-1)^{N+1}t^{2N+1}s_N(t),
\end{equation}
where $s_N(t)$ is given by
\begin{equation}\label{asy-thm-pf2}
s_N(t)=\frac4{\pi^{2N}}\sum_{k=0}^\infty\frac{1}{(t^2+\pi^2(2k+1)^2)(2k+1)^{2N}}.
\end{equation}
Suppose for a moment that $q$ is real and positive. Then substituting the right-hand side of (\ref{asy-thm-pf1}) into (\ref{rem}), we get
\begin{equation}\label{asy-thm-pf3}
\begin{aligned}
R_N(s,q)&=\frac{q^s}{\Gamma(s)}\frac4{\pi^{2N}}\int_0^\infty
\sum_{k=0}^\infty\frac{(-1)^{N+1}t^{2N+1}}{(t^2+\pi^2(2k+1)^2)(2k+1)^{2N}}t^se^{-qt}\frac{\d t}t \\
&=\frac{q^s}{\Gamma(s)}\frac4{\pi^{2N}}\sum_{k=0}^\infty\int_0^\infty
\frac{(-1)^{N+1}t^{2N+1}}{(t^2+\pi^2(2k+1)^2)(2k+1)^{2N}}t^se^{-qt}\frac{\d t}t \\
&=(-1)^{N+1}\frac{4}{\Gamma(s)}\frac{1}{\pi^{2N+2}q^{2N+1}}\sum_{k=0}^\infty\int_0^\infty
\frac{u^{s+2N+1}}{1+\left(\frac u{\pi q}\right)^2}\frac{e^{-(2k+1)u}}{(2k+1)^{1-s}}\frac{\d u}u,
\end{aligned}
\end{equation}
where the third equality is obtained by changing the integration variable from $t$ to $u$ with $u=qt/(2k+1)$
(see \cite[p. 4, (2.3)]{Ne}).
It needs to mention that in the above reasoning, changing the orders of summation and integration is allowed  because of  the absolute convergence of the series.
The following integral representation can be found in \cite[p. 12, (A 2)]{Ne}
\begin{equation}\label{pi-def}
{\it\Pi}_p(w)=\frac1{\Gamma(p)}\int_0^\infty\frac{t^{p}e^{-t}}{1+\left(\frac tw\right)^2}\frac{\d t}t,
\end{equation}
which is valid for $|\arg w|<\pi/2$ and Re$(p)>0.$ Moreover, when $w\neq1$ and $p$ are positive, we find that (see \cite[p. 12, Proposition A.1]{Ne})
\begin{equation}\label{pi-pro}
0<{\it\Pi}_p(w)<1.
\end{equation}
Also, by using (\ref{pi-def}) and making a change of variables from $u$ to $v$ with $v=(2k+1)u$ in the third equality of (\ref{asy-thm-pf3}), we deduce that 
\begin{equation}\label{int-lem}
\begin{aligned}
R_N(s,q)
&=\frac{(-1)^{N+1}4}{\pi^{2N+2}q^{2N+1}}\frac{1}{\Gamma(s)}
\sum_{k=0}^\infty\int_0^\infty\frac{\left(\frac{v}{2k+1}\right)^{s+2N}e^{-v}}{1+\left(\frac{v}{(2k+1)\pi q}\right)^2}
\frac{1}{(2k+1)^{2-s}} \d v \\
&=\frac{(-1)^{N+1}4}{\pi^{2N+2}q^{2N+1}}\frac{1}{\Gamma(s)}
\sum_{k=0}^\infty\int_0^\infty\frac{v^{s+2N+1}e^{-v}}{1+\left(\frac{v}{(2k+1)\pi q}\right)^2}\frac{1}{(2k+1)^{2N+2}}\frac{\d v}v \\
&=\frac{(-1)^{N+1}4}{\pi^{2N+2}q^{2N+1}}\frac{\Gamma(s+2N+1)}{\Gamma(s)}
\sum_{k=0}^\infty\frac{{\it\Pi}_{s+2N+1}((2k+1)\pi q)}{(2k+1)^{2N+2}}
\end{aligned}
\end{equation}
(see \cite[p. 5, (2.5)]{Ne}), and the proof for the first part is complete. 
Now we go to the second part. 
From \cite[p. 824]{KS} and (\ref{E-B}) we obtain the identity
\begin{equation}\label{beta-euler}
\sum_{k=0}^\infty\frac1{(2k+1)^{2N}}=\frac{(-1)^N\pi^{2N}E_{2N-1}(0)}{4(2N-1)!}
\end{equation}
for $N\in\mathbb N.$
Therefore, by using (\ref{pi-pro}), (\ref{int-lem}) and (\ref{beta-euler}),  we can assert that (see \cite[p. 6]{Ne})
\begin{equation}\label{pf-final}
\begin{aligned}
R_N(s,q)
&=\frac{(-1)^{N+1}4(s)_{2N+1}}{\pi^{2N+2}q^{2N+1}}
\sum_{k=0}^\infty\frac{{\it\Pi}_{s+2N+1}((2k+1)\pi q)}{(2k+1)^{2N+2}} \\
&=\frac{(-1)^{N+1}4(s)_{2N+1}}{\pi^{2N+2}q^{2N+1}}
\sum_{k=0}^\infty\frac{1}{(2k+1)^{2N+2}}\theta_N(s,q) \\
&=\frac{E_{2N+1}(0)}{(2N+1)!}\frac{(s)_{2N+1}}{q^{2N+1}}\theta_N(s,q),
\end{aligned}
\end{equation}
where $0<\theta_N(s,q)<1$ is a suitable number which depends on $s,q$ and $N.$
This completes the proof of Theorem \ref{asy-thm}.

\subsection*{Proof of Theorem \ref{asy-pro2}.}
The proof follows along the same lines as the proofs of (1.7) and (1.8) in \cite[Theorem 1.2]{Ne}.

First we prove (\ref{pro2-1}). Suppose that $N\in\mathbb N$ and $s$ is an arbitrary complex number with
Re$(s)>-2N,$ from (\ref{beta-euler}) and (\ref{pf-final}), 
\begin{equation}\label{pro2-pf1}
\begin{aligned}
|R_N(s,q)|&\leq
4\frac{|(s)_{2N+1}|}{\pi^{2N+2}q^{2N+1}}\sum_{k=0}^\infty\frac{|{\it\Pi}_{s+2N+1}((2k+1)\pi q)|}{(2k+1)^{2N+2}} \\
&\leq
4\frac{|(s)_{2N+1}|}{\pi^{2N+2}q^{2N+1}}\sum_{k=0}^\infty\frac{1}{(2k+1)^{2N+2}} \sup_{r\geq1}|{\it\Pi}_{s+2N+1}((2r+1)\pi q)| \\
&=\frac{|E_{2N+1}(0)|}{(2N+1)!}\frac{|(s)_{2N+1}|}{|q|^{2N+1}}\sup_{r\geq1}|{\it \Pi}_{s+2N+1}((2r+1)\pi q|.
\end{aligned}
\end{equation}
This finished the proof of (\ref{pro2-1}).

The bound (\ref{pro2-2}) can be shown as follows. 
For any positive real $t$ and $|\arg q|<\pi$, we have (see \cite[p. 6, (3.2) and (3.3)]{Ne}),
\begin{equation}\label{ne-32}
\frac1{|(q+t)^{s+2N+2}|}\leq \frac{1}{|q+t|^{{\rm Re}(s)+2N+2}}\max\left(1,e^{{\rm Im}(s)\arg q}\right)
\end{equation}
and
\begin{equation}\label{ne-33}
|q+t|^2\geq(t+|q|)^2\cos^2\left(\frac{\arg q}{2}\right).
\end{equation}
Also, by (\ref{EF-ex}), we obtain
\begin{equation}\label{ne-34}
\begin{aligned}
(-1)^N(E_{2N+1}(0)-\E_{2N+1}(-t))&=\frac{4(2N+1)!}{\pi^{2N+2}}\sum_{k=0}^\infty\frac{1-\cos((2k+1)\pi t)}{(2k+1)^{2N+2}}\geq0
\end{aligned}
\end{equation}
for any $t>0.$ Hence if $N\geq1$ is fixed, then the function $E_{2N+1}(0)-\E_{2N+1}(-t)$ does not change the sign.
In the following, we follow the arguments of \cite[(3.5)]{Ne} by considering (\ref{parts}), (\ref{ne-32}), (\ref{ne-33}) and (\ref{ne-34}).
From Theorem \ref{asy-pro1},  we have  
\begin{equation}\label{ne-35}
\begin{aligned}
|R_N(s,q)|&\leq
\frac{|(s)_{2N+2}|}{(2N+1)!}|q^s|\int_0^\infty \frac{|E_{2N+1}(0)-\E_{2N+1}(-t)|}{|(q+t)^{s+2N+2}|} \d t \\
&\leq
\frac{|(s)_{2N+2}|}{(2N+1)!}|q^s|\left| \int_0^\infty \frac{E_{2N+1}(0)-\E_{2N+1}(-t)}{|q+t|^{{\rm Re}(s)+2N+2}} \d t \right|
\max\left(1,e^{{\rm Im}(s)\arg q}\right)
\\
&\leq
\frac{|(s)_{2N+2}|}{(2N+1)!}|q^s|\left| \int_0^\infty \frac{E_{2N+1}(0)-\E_{2N+1}(-t)}{(|q|+t)^{{\rm Re}(s)+2N+2}} \d t \right|  \\
&\quad\times\sec^{{\rm Re}(s)+2N+2}\left(\frac{\arg q}{2}\right)\max\left(1,e^{{\rm Im}(s)\arg q}\right)
\\
&=\frac{|(s)_{2N+2}|}{({\rm Re}(s))_{2N+2}}\frac{|q^s|}{|q|^{{\rm Re}(s)}}|R_N({\rm Re}(s),|q|)|\sec^{{\rm Re}(s)+2N+2}\left(\frac{\arg q}{2}\right) \\
&\quad\times\max\left(1,e^{{\rm Im}(s)\arg q}\right) \\
&=\frac{|(s)_{2N+2}|}{({\rm Re}(s))_{2N+2}}|R_N({\rm Re}(s),|q|)|\sec^{{\rm Re}(s)+2N+2}\left(\frac{\arg q}{2}\right) \\
&\quad\times\max\left(1,e^{-{\rm Im}(s)\arg q}\right).
\end{aligned}
\end{equation}
And by Theorem \ref{asy-thm}, the factor $|R_N({\rm Re}(s),|q|)|$ may be bounded as follows
\begin{equation}\label{ne-36}
\begin{aligned}
|R_N({\rm Re}(s),|q|)|&\leq \frac{|E_{2N+1}(0)|}{(2N+1)!}\frac{({\rm Re}(s))_{2N+1}}{|q|^{2N+1}} \\
&= \frac{|E_{2N+1}(0)|}{(2N+1)!}\frac{|(s)_{2N+1}|}{|q|^{2N+1}}\frac{({\rm Re}(s))_{2N+2}}{|(s)_{2N+2}|}
\frac{|s+2N+1|}{{\rm Re}(s)+2N+1}
\end{aligned}
\end{equation}
(see \cite[p. 7]{Ne}).
Substituting this inequality into (\ref{ne-35}) yields the desired result.

\subsection*{Proof of Theorem \ref{thm2}.}

To prove this, we need the following identity on the Pochhammer symbol,
which we will be often used in the subsequent:
\begin{equation}\label{second}
\frac{\d}{\d s}\left(\frac{(s)_k}{k!}\right)=\sum_{j=0}^{k-1}\frac{(s)_j}{j!(k-j)},
\end{equation}
where $k\in\mathbb N_0$ (see \cite[p. 3223, (4)--(5)]{EE0} and \cite[p. 25, (2.7)--(2.8)]{EE1}).

Now we are at the position to prove Theorem \ref{thm2}. 
By fixing $s,q\in\mathbb C$ and for $t\in\mathbb R,$ we consider the auxiliary function
$$f(t)=\frac{\partial}{\partial s}\zeta_E(s,q+t),$$
which satisfies all the assumption of Theorem \ref{thm2}
(see \cite[p. 2833]{Ru}).
Then applying (\ref{J-1}), we immediately get
\begin{equation}\label{zeta-p-1}
\begin{aligned}
\Delta_+(t)&=\frac{\partial}{\partial s}\left(\zeta_E(s,q+t+1)+\zeta_E(s,q+t)\right)=\frac{\partial}{\partial s}(q+t)^{-s} \\
&=\frac{\partial}{\partial s}\left(e^{-s\log(q+t)}\right)=-(q+t)^{-s}\log(q+t)
\end{aligned}
\end{equation}
and setting $t=0$ in (\ref{zeta-p-1}),we have
\begin{equation}\label{zeta-p-2}
\begin{aligned}
\Delta_+(0)=-q^{-s}\log q.
\end{aligned}
\end{equation}
By (\ref{zeta-p-1})  
\begin{equation}\label{zeta-p-3}
\begin{aligned}
\Delta_+'(t)=-(q+t)^{-s-1}+s(q+t)^{-s-1}\log(q+t).
\end{aligned}
\end{equation}
Then letting  $t=0$ in (\ref{zeta-p-3}), we obtain
\begin{equation}\label{zeta-p-4}
\begin{aligned}
\Delta_+'(0)=-q^{-s-1}+sq^{-s-1}\log q.
\end{aligned}
\end{equation}
Similarly, we have
\begin{equation}\label{zeta-p-5}
\begin{aligned}
\Delta_+^{(2)}(t)&=(s+1)(q+t)^{-s-2}+s(q+t)^{-s-2} \\
&\quad-s(s+1)(q+t)^{-s-2}\log(q+t)
\end{aligned}
\end{equation}
and
\begin{equation}\label{zeta-p-6}
\begin{aligned}
\Delta_+^{(3)}(t)&=-(z+1)(z+2)(q+t)^{-z-3}-z(z+2)(q+t)^{-z-3} \\
&\quad-z(z+1)(q+t)^{-z-3}+z(z+1)(z+2)(q+t)^{-z-3}\log(q+t).
\end{aligned}
\end{equation}
This suggests the following general formula for positive integers $k\geq2,$
\begin{equation}\label{zeta-p-k}
\begin{aligned}
\Delta_+^{(k)}(t)&=(-1)^k\sum_{j=0}^{k-1}\frac{(s)_k}{s+j}(q+t)^{-s-k} \\
&\quad+(-1)^{k+1}(s)_k(q+t)^{-s-k}\log(q+t).
\end{aligned}
\end{equation}
It is easy to see that
\begin{equation}\label{sum-co}
\sum_{j=0}^{k-1}\frac{(s)_k}{s+j}=\frac{\d}{\d s}((s)_k),
\end{equation}
so by \eqref{zeta-p-k} and \eqref{sum-co}, we have
\begin{equation}\label{zeta-p-k-1}
\begin{aligned}
\Delta_+^{(k)}(t)
&=(-1)^k\frac{\d}{\d s}((s)_k)(q+t)^{-s-k} \\
&\quad+(-1)^{k+1}(s)_k(q+t)^{-s-k}\log(q+t) \\
&=(-1)^kk!\sum_{j=0}^{k-1}\frac{(s)_j}{j!(k-j)}(q+t)^{-s-k} \\
&\quad+(-1)^{k+1}(s)_k(q+t)^{-s-k}\log(q+t),
\end{aligned}
\end{equation}
the last equality follows from \eqref{second}. Setting $t=0$ in (\ref{zeta-p-k-1}), we get
\begin{equation}\label{zeta-p-k-1-0}
\begin{aligned}
\Delta_+^{(k)}(0)
&=(-1)^kk!\sum_{j=0}^{k-1}\frac{(s)_j}{j!(k-j)}q^{-s-k} +(-1)^{k+1}(s)_kq^{-s-k}\log q.
\end{aligned}
\end{equation}
From (\ref{zeta-p-2}), (\ref{zeta-p-4}), and (\ref{zeta-p-k-1-0}) with $k$ replaced by $2k+1,$
 \eqref{E-aym}  implies that
\begin{equation}\label{zeta-fi-dir}
\begin{aligned}
\zeta_E'(s,q)&\sim-\left(\frac12\log q\right)q^{-s}+\frac14\left(1- s\log q\right)q^{-s-1} \\
&\quad-\frac12\sum_{k=1}^\infty E_{2k+1}(0)\sum_{j=0}^{2k}\frac{(s)_j}{j!(2k-j+1)}q^{-s-2k-1} \\
&\quad+\frac12\sum_{k=1}^\infty E_{2k+1}(0)\frac{(s)_{2k+1}}{(2k+1)!}q^{-s-2k-1}\log q
\end{aligned}
\end{equation}
provided that $|q|\to\infty$ in the sector $|\arg q|<\pi.$
From Theorem \ref{thm1},  the second infinite sum of the right-hand side of \eqref{zeta-fi-dir} has an asymptotic expansion
\begin{equation}\label{thm1-eq}
\left(\frac12 q^{-s}+\frac14 s q^{-s-1}-\zeta_E(s,q)\right)\log q.
\end{equation}
Substituting  (\ref{thm1-eq}) into (\ref{zeta-fi-dir}), as $|q|\to\infty$ in the sector $|\arg q|<\pi,$ we obtain
$$\zeta_E'(s,q)\sim\frac14 q^{-s-1}-\zeta_E(s,q)\log q-\frac12\sum_{k=1}^\infty E_{2k+1}(0)\sum_{j=0}^{2k}
\frac{(s)_j}{j!(2k-j+1)}q^{-s-2k-1},$$
which completes the proof.

\subsection*{Proof of Corollary \ref{n=01}.}
Setting $s=-n$ for integer $n\in\mathbb N_0$ in Theorem \ref{thm2}. Using Corollary \ref{cor1} and
$$(-n)_j=(-1)^j\frac{\Gamma(n+1)}{\Gamma(n+1-j)}=(-1)^j\frac{n!}{(n-j)!},$$
we obtain the asymptotic expansion
\begin{equation}\label{d-sum}
\begin{aligned}
\zeta_E'(-n,q)&\sim\frac14 q^{n-1}-\frac12E_n(q)\log q \\
&\quad-\frac12\sum_{k=1}^\infty E_{2k+1}(0)\sum_{j=0}^{2k}\binom nj\frac{(-1)^j}{2k-j+1}q^{n-2k-1}
\\
&\sim\frac14 q^{n-1}-\frac12E_n(q)\log q \\
&\quad-\frac12\sum_{k=1}^\infty E_{2k+1}(0)\sum_{j=0}^{\min(n,2k)}\binom nj\frac{(-1)^j}{2k-j+1}q^{n-2k-1},
\end{aligned}
\end{equation}
as $|q|\to\infty$ provided that $|\arg q|<\pi.$
We have the desired result.

\subsection*{Proof of Corollary \ref{cor2}.}
Let $n\geq2.$ It is easy to see that the following double sum identity holds
\begin{equation}\label{s-sum}
\sum_{k=1}^\infty \sum_{j=0}^{\min(n,2k)}=\sum_{k=1}^{\left\lfloor \frac{n}{2}\right\rfloor}\sum_{j=0}^{2k}
+\sum_{k=\left\lfloor \frac{n}{2}\right\rfloor+1}^\infty\sum_{j=0}^n
\end{equation}
for $\lfloor x\rfloor =\max\{m\in \mathbb {Z} \mid m\leq x\}.$
Furthermore, it is interesting to notice that
\begin{equation}\label{beta-like}
\begin{aligned}
\sum_{j=0}^n\binom nj\frac{(-1)^j}{j+s}
&=\sum_{j=0}^n\binom nj(-1)^j\int_0^1t^{s+j-1}\d t =\frac{\Gamma(s)\Gamma(n+1)}{\Gamma(s+n+1)} \\
&=\frac{n!}{s(s+1)(s+2)\cdots(s+n)}.
\end{aligned}
\end{equation}
Setting $s=-k~(k>n)$ in \eqref{beta-like}, we obtain (see \cite[p. 2833, (16)]{Ru})
\begin{equation}\label{beta-like2}
\begin{aligned}
\sum_{j=0}^n\binom nj\frac{(-1)^j}{k-j}
&=\frac{(-1)^{n}n!}{k(k-1)(k-2)\cdots(k-n)}.
\end{aligned}
\end{equation}
Now the result follows  from (\ref{d-sum}), (\ref{s-sum}) and \eqref{beta-like2}. 

\subsection*{Proof of Theorem \ref{thm3}.}
First, we extend the proof of Theorem \ref{thm2} to obtain an asymptotic expansion which represents the higher order derivatives of the alternating Hurwitz (or Hurwitz-type Euler) zeta function in terms of the lower orders.
Recall our notation
\begin{equation}\label{any-der}
\zeta_{E}^{(m)}(s,q)\equiv\frac{\partial^m}{\partial s^m}\zeta_E(s,q)
\end{equation}
for $m\geq2.$
Now, repeating, for  $\zeta_E''(s,q),$ the same procedure just used for $\zeta_E'(s,q)$ in the proof of Theorem \ref{thm2}, we arrive at the following expansion
\begin{equation}\label{thm1-eq-re}
\begin{aligned}
\zeta_E''(s,q)&\sim-\frac14 q^{-s-1}\log q-\zeta_E'(s,q)\log q \\
&\quad+\frac12\sum_{k=1}^\infty E_{2k+1}(0)\sum_{j_1=0}^{2k}\frac{(s)_{j_1}}{j_1!(2k-j_1+1)}q^{-s-2k-1}\log q \\
&\quad-\frac12\sum_{k=1}^\infty E_{2k+1}(0)\sum_{j_1=0}^{2k}\frac{1}{2k-j_1+1}
\sum_{j_2=0}^{j_1-1}\frac{(s)_{j_2}}{j_2!(j_1-j_2)}q^{-s-2k-1}.
\end{aligned}
\end{equation}
In fact, the above identity (\ref{thm1-eq-re}) follows directly from Theorem \ref{thm2} and \eqref{second}.
Note that by Theorem \ref{thm2},
the first infinite sum on the right-hand side of \eqref{thm1-eq-re} has an asymptotic expansion
\begin{equation}\label{thm2-eq}
\left(\frac14 q^{-s-1}-\zeta_E(s,q)\log q-\zeta_E'(s,q)\right)\log q.
\end{equation}
Substituting (\ref{thm2-eq}) into  (\ref{thm1-eq-re}) and using (\ref{ckm}), we get
\begin{equation}\label{two-comp}
\begin{aligned}
\zeta_E''(s,q)&\sim-2\zeta_E'(s,q)\log q -\zeta_E(s,q)\log^2 q \\
&\quad-\frac12\sum_{k=1}^\infty E_{2k+1}(0)\sum_{j_1=0}^{2k}\frac{1}{2k-j_1+1}
\sum_{j_2=0}^{j_1-1}\frac{(s)_{j_2}}{j_2!(j_1-j_2)}q^{-s-2k-1} \\
&\sim-2\zeta_E'(s,q)\log q -\zeta_E(s,q)\log^2 q-\frac12\Sigma_2(s,q)
\end{aligned}
\end{equation}
provided that $|q|\to\infty$ and $|\arg q|<\pi.$

With some additional effort, the following general recurrence can be derived for $m\geq 2$, which yields an asymptotic expansion for the higher order derivatives
of the alternating Hurwitz zeta function in terms of the lower orders: 
\begin{equation}\label{m-der-pf}
\zeta_E^{(m)}(s,q)\sim-\sum_{j=1}^m\binom mj\zeta_E^{(m-j)}(s,q)\log^j q-\frac12\Sigma_m(s,q)
\end{equation}
provided that $|q|\to\infty$ in the sector $|\arg q|<\pi,$
where $\Sigma_m(s,q)$ being given by (\ref{ckm}).

The proof of (\ref{m-der-pf}) proceeds by an induction on $m.$
The case $m=2$ is clear by (\ref{two-comp}).
Also, by (\ref{ckm}) and \eqref{second}  it is easily seen that
\begin{equation}\label{pf-1}
\frac{\partial}{\partial s}
c_{k}^m(s)=c_{k}^{m+1}(s).
\end{equation}
Applying the inductive hypothesis, (\ref{m-der-pf}) and (\ref{pf-1}), we obatin
$$
\begin{aligned}
\zeta_E^{(m+1)}(s,q)
&\sim-\frac{\partial}{\partial s}\left(\sum_{j=1}^m\binom mj\zeta_E^{(m-j)}(s,q)\log^j q+\frac12\Sigma_m(s,q)\right) \\
&\quad(\text{by the inductive hypothesis (\ref{m-der-pf})}) 
\\
&\sim-\sum_{j=1}^m\binom mj\zeta_E^{(m-j+1)}(s,q)\log^j q +\frac12\sum_{k=1}^\infty  c_{k}^{m}(s)q^{-s-2k-1}\log q \\
&\quad-\frac12\sum_{k=1}^\infty  \frac{\partial}{\partial s}\left(c_{k}^m(s)\right)q^{-s-2k-1}
\\
&\sim-\sum_{j=1}^m\binom mj\zeta_E^{(m-j+1)}(s,q)\log^j q - \sum_{j=1}^m\binom mj\zeta_E^{(m-j)}(s,q)\log^{j+1} q \\
&\quad-\zeta_E^{(m)}(s,q)\log q-\frac12\sum_{k=1}^\infty  \frac{\partial}{\partial s}\left(c_{k}^m(s)\right)q^{-s-2k-1} 
\\
&\quad(\text{by using (\ref{m-der-pf}) again}) 
\\
&\sim-\binom{m+1}1\zeta_E^{(m)}(s,q)\log q \\
&\quad-\sum_{j=1}^{m-1}\left(\binom{m}{j+1} +\binom m{j}\right) \zeta_E^{(m-j)}(s,q)\log^{j+1}q  \\
&\quad -\binom{m+1}{m+1}\zeta_E(s,q)\log^{m+1} q
-\frac12\sum_{k=1}^\infty c_{k}^{m+1}(s)q^{-s-2k-1} \\
&\quad(\text{by  (\ref{pf-1})})
\\
&\sim-\sum_{j=1}^{m+1}\binom {m+1}j\zeta_E^{(m-j+1)}(s,q)\log^j q-\frac12\sum_{k=1}^\infty  c_{k}^{m+1}(s)q^{-s-2k-1}
\end{aligned}
$$
and this completes the inductive argument.

\subsection*{Proof of Theorem \ref{new-series}.}

Suppose that $\{N_k\}_{k=0}^\infty$ is an arbitrary sequence of positive integers and $s$ is any complex number such that 
Re$(s)>-2N_k$ for each $k\in\mathbb N_0.$
When $N=1,$ (\ref{int-lem}) is equivalent to
\begin{equation}\label{fin-rem0}
R_1(s,q)=\frac{4}{\pi^{4}q^{3}}\frac{1}{\Gamma(s)}\sum_{k=0}^\infty \frac{1}{(2k+1)^{4}}
\int_0^\infty\frac{v^{s+3}e^{-v}}{1+\left(\frac{v}{(2k+1)\pi q}\right)^2}
\frac{\d v}v.
\end{equation}
Since we have the geometrical series
\begin{equation}\label{fin-rem1}
\begin{aligned}
\frac1{1+\left(\frac{v}{(2k+1)\pi q}\right)^2}&=\sum_{n=1}^{N_k-1}(-1)^{n-1}\left(\frac{v}{(2k+1)\pi q}\right)^{2n-2} \\
&\quad+(-1)^{N_k-1}\frac1{1+\left(\frac{v}{(2k+1)\pi q}\right)^2}\left(\frac{v}{(2k+1)\pi q}\right)^{2N_k-2},
\end{aligned}
\end{equation}
substituting  the right-hand side of (\ref{fin-rem1}) into (\ref{fin-rem0}) we get
\begin{equation}\label{fin-rem2}
\begin{aligned}
R_1(s,q)&=4\sum_{k=0}^\infty\sum_{n=1}^{N_k-1}(-1)^{n-1}
\frac{(s)_{2n+1}}{((2k+1)\pi)^{2n+2}q^{2n+1}} \\
&\quad+4\sum_{k=0}^\infty(-1)^{N_k-1}\frac{(s)_{2N_k+1}{\it\Pi}_{s+2N_k+1}\left({(2k+1)\pi q}\right)}{((2k+1)\pi)^{2N_k+2}q^{2N_k+1}},
\end{aligned}
\end{equation}
which we have used the equation (\ref{pi-def}).
From (\ref{rem-te1}) and (\ref{fin-rem2}), we deduce a new series representation of the alternating Hurwitz zeta function 
$\zeta_E(s,q)$:
\begin{equation}\label{ex-asym}
\begin{aligned}
\zeta_E(s,q)&=\frac12{q^{-s}}+\frac14sq^{-s-1}-2q^{-s}\left( \sum_{k=0}^\infty\sum_{n=1}^{N_k-1}(-1)^{n-1}
\frac{(s)_{2n+1}}{((2k+1)\pi)^{2n+2}q^{2n+1}} \right. \\
&\quad+\left. \sum_{k=0}^\infty(-1)^{N_k-1}\frac{(s)_{2N_k+1}{\it\Pi}_{s+2N_k+1}\left({(2k+1)\pi q}\right)}{((2k+1)\pi)^{2N_k+2}q^{2N_k+1}} \right),
\end{aligned}
\end{equation}
as long as $|\arg q|<\pi.$ This completes the proof.

\section{Special cases and examples}

In this section, as  examples, we shall write the asymptotic series for the second derivative of $\zeta_{E}(s,q)$ at  $s=-n$ explicitly as $|q|\to\infty$ in the sector $|\arg q|<\pi.$

For an integer $n\in\mathbb N_0,$ by Corollary \ref{cor3} with $m=2$ we have
\begin{equation}\label{k=2}
\begin{aligned}
\zeta_E^{''}(-n,q)
&\sim2\zeta_E'(-n,q)\log q-\zeta_E(-n,q)\log^2q \\
&\quad-\frac12\sum_{k=1}^\infty E_{2k+1}(0)\sum_{j_1=0}^{2k}\frac1{2k-j_1+1}  \sum_{j_2=0}^{\min(n,j_1-1)} \binom{n}{j_2}
\frac{(-1)^{j_2}}{j_1-j_2}q^{n-2k-1}.
\end{aligned}
\end{equation}
Now setting $n=0,1$ in (\ref{k=2}) respectively, using Corollary \ref{n=01}, we have
\begin{equation}\label{k=2 with n=0}
\begin{aligned}
\zeta_E^{''}(0,q)
&\sim\frac12\log^2 q -\left(\frac12\log q\right)q^{-1}  \\
&\quad+\sum_{k=1}^\infty E_{2k+1}(0)\left(\frac{\log q}{2k+1}-\frac12\sum_{j=1}^{2k}\frac1{j(2k-j+1)}\right)q^{-2k-1} ,
\end{aligned}
\end{equation}
and
\begin{equation}\label{k=2 with n=1}
\begin{aligned}
\zeta_E^{''}(-1,q)
&\sim-\left(\frac12\log q-\frac14\log^2 q\right) +\left(\frac12\log^2q\right)q\\
&\quad+\frac12\sum_{k=1}^\infty E_{2k+1}(0)\left(\sum_{j=2}^{2k}\frac1{(2k-j+1)j(j-1)} \right. \\
&\quad\quad\qquad\qquad\qquad\qquad\qquad\left.-\frac1{2k}-\frac{2\log q}{(2k+1)(2k)} \right)q^{-2k}.
\end{aligned}
\end{equation}
Using \eqref{s-sum}, we can easily check that
\begin{equation}\label{zp(2)-1}
\begin{aligned}
\sum_{k=1}^\infty \sum_{j_1=0}^{2k}\sum_{j_2=0}^{\min(n,j_1-1)}
=\sum_{k=1}^{\left\lfloor \frac{n}{2}\right\rfloor}\sum_{j_1=0}^{2k}\sum_{j_2=0}^{j_1-1}
+\sum_{k=\left\lfloor \frac{n}{2}\right\rfloor+1}^\infty\sum_{j_1=0}^{2k}\sum_{j_2=0}^{\min(n,j_1-1)}.
\end{aligned}
\end{equation}
For  $n\geq2,$ from (\ref{k=2}) and (\ref{zp(2)-1}) we obtain
\begin{equation}\label{k=2-ne}
\begin{aligned}
\zeta_E^{''}(-n,q)
&\sim-2\zeta_E'(-n,q)\log q-\zeta_E(-n,q)\log^2q \\
&\quad-\frac12\sum_{k=1}^{\left\lfloor \frac{n}{2}\right\rfloor} E_{2k+1}(0)\sum_{j_1=0}^{2k}\frac1{2k-j_1+1}\sum_{j_2=0}^{j_1-1} \binom{n}{j_2}
\frac{(-1)^{j_2}}{j_1-j_2}q^{n-2k-1} \\
&\quad-\frac12\sum_{k=\left\lfloor \frac{n}{2}\right\rfloor+1}^\infty E_{2k+1}(0)\sum_{j_1=0}^{2k}\frac1{2k-j_1+1} \\
&\qquad\times\sum_{j_2=0}^{\min(n,j_1-1)} \binom{n}{j_2}
\frac{(-1)^{j_2}}{j_1-j_2}q^{n-2k-1}.
\end{aligned}
\end{equation}
By a slight modification of \eqref{s-sum}, we also have
\begin{equation}\label{zp(2)-2}
\begin{aligned}
\sum_{k=\left\lfloor \frac{n}{2}\right\rfloor+1}^\infty\sum_{j_1=0}^{2k}\sum_{j_2=0}^{\min(n,j_1-1)}
=\sum_{k=\left\lfloor \frac{n}{2}\right\rfloor+1}^\infty\sum_{j_1=0}^{n}\sum_{j_2=0}^{j_1-1}
+\sum_{k=\left\lfloor \frac{n}{2}\right\rfloor+1}^\infty\sum_{j_1=n+1}^{2k}\sum_{j_2=0}^{n},
\end{aligned}
\end{equation}
this can be rewritten  (\ref{k=2-ne}) as
\begin{equation}\label{k=2-ne1}
\begin{aligned}
\zeta_E^{''}(-n,q)
&\sim-2\zeta_E'(-n,q)\log q-\zeta_E(-n,q)\log^2q \\
&\quad-\frac12\sum_{k=1}^{\left\lfloor \frac{n}{2}\right\rfloor} E_{2k+1}(0)\sum_{j_1=0}^{2k}\frac1{2k-j_1+1}\sum_{j_2=0}^{j_1-1} \binom{n}{j_2}
\frac{(-1)^{j_2}}{j_1-j_2}q^{n-2k-1} \\
&\quad-\frac12\sum_{k=\left\lfloor \frac{n}{2}\right\rfloor+1}^\infty E_{2k+1}(0)\sum_{j_1=0}^{n}\frac1{2k-j_1+1}\sum_{j_2=0}^{j_1-1} \binom{n}{j_2}
\frac{(-1)^{j_2}}{j_1-j_2}q^{n-2k-1}  \\
&\quad-\frac12\sum_{k=\left\lfloor \frac{n}{2}\right\rfloor+1}^\infty E_{2k+1}(0)
\sum_{j_1=n+1}^{2k}\frac{(-1)^n}{(2k-j_1+1)(n+1)\binom{j_1}{n+1}}
q^{n-2k-1},
\end{aligned}
\end{equation}
which we have used the combinational identity
$$\sum_{j_2=0}^{n} \binom{n}{j_2}\frac{(-1)^{j_2}}{j_1-j_2}=\frac{(-1)^nn!}{j_1(j_1-1)\cdots(j_1-n)}=
\frac{(-1)^n}{(n+1)\binom{j_1}{n+1}}$$
(see \eqref{beta-like2}).
When $n\geq2,$ substituting the results of Corollaries \ref{cor1} and  \ref{cor2} into (\ref{k=2-ne1}) and using \eqref{beta-like2} we get
\begin{equation}\label{k=2-fin}
\begin{aligned}
&\zeta_E^{''}(-n,q) \\
&\sim-\frac12 q^{n-1}\log q +\frac12 E_n(q)\log^2 q \\
&\quad+\sum_{k=1}^{\left\lfloor \frac{n}{2}\right\rfloor} E_{2k+1}(0)\sum_{j_1=0}^{2k}
\left(\binom{n}{j_1}\frac{(-1)^{j_1}\log q}{2k-j_1+1} -\frac1{2(2k-j_1+1)}\sum_{j_2=0}^{j_1-1} \binom{n}{j_2}
\frac{(-1)^{j_2}}{j_1-j_2}\right) \\
&\qquad\times q^{n-2k-1} \\
&\quad+\sum_{k=\left\lfloor \frac{n}{2}\right\rfloor+1}^\infty E_{2k+1}(0)\left(
\frac{(-1)^n\log q}{(n+1)\binom{2k+1}{n+1}}
-\frac12\sum_{j_1=0}^{n}\frac1{2k-j_1+1}\sum_{j_2=0}^{j_1-1} \binom{n}{j_2}
\frac{(-1)^{j_2}}{j_1-j_2}\right. \\
&\quad\qquad\qquad\qquad\qquad\qquad\left.-\frac12\sum_{j_1=n+1}^{2k}\frac{(-1)^n}{(2k-j_1+1)(n+1)\binom{j_1}{n+1}}\right)
q^{n-2k-1} .
\end{aligned}
\end{equation}
This result may be regarded as an analogue of  \cite[(170)]{EE}. And the above asymptotic expansions turn out to be very useful in the effective Lagrangian theory of quark confinement (see \cite{EE}, \cite{EE0} and \cite{EE1}).

When $n=2,3,$ (\ref{k=2-fin}) yields the following asymptotic formulas:
$$
\begin{aligned}
\zeta_E^{''}(-2,q)
&\sim\left(\frac12\log^2 q\right)q^2-\left(\frac12\log q+\frac12\log^2 q\right)q -\left(\frac38-\frac{1}{12}\log q\right)q^{-1} \\
&\quad+\sum_{k=2}^\infty E_{2k+1}(0)\left(
\frac{2\log q}{(2k+1)(2k)(2k-1)} \right.\\
&\qquad\qquad\qquad\qquad\left.-\frac12\left(\frac1{2k}+\frac{5}{2(2k-1)}\right)\right. \\
&\qquad\qquad\qquad\qquad-\left.\sum_{j=3}^{2k}\frac1{(2k-j+1)j(j-1)(j-2)} \right)q^{-2k+1},
\end{aligned}
$$
$$
\begin{aligned}
\zeta_E^{''}(-3,q)
&\sim\left(\frac12\log^2 q\right)q^3-\left(\frac12\log q+\frac34\log^2 q\right)q^2
\\
&\quad +\left(\frac14+\frac{11}{24}\log q+\frac18\log^2 q\right) +\left(\frac1{48}+\frac1{40}\log q\right)q^{-2} \\
&\quad+\sum_{k=3}^\infty E_{2k+1}(0)\left(
\sum_{j=4}^{2k}\frac3{(2k-j+1)j(j-1)(j-2)(j-3)} \right. \\
&\qquad\qquad\qquad\qquad
-\frac{6\log q}{(2k+1)(2k)(2k-1)(2k-2)}  \\
&\qquad\qquad\qquad\qquad
\left. -\frac12\left(\frac1{2k}-\frac5{2(2k-1)}+\frac{11}{6(2k-2)}\right)
\right)q^{-2k+2} .
\end{aligned}
$$

\section{Concluding remarks and discussions}

For Re$(s)>0$, the Riemann zeta function $\zeta(s)$ and the alternating zeta function $\zeta_E(s)$ are connected by the equation (\ref{Riemann-Euler1}).
Furthermore, the function $\zeta_E(s)$ appeared as a bridge for Euler's derivation of the functional equation for $\zeta(s).$
In fact, according to Weil's history \cite[pp.~273--276]{Weil},
Euler  ``proved"
\begin{equation}~\label{fe}
	\frac{\zeta_E(1-s)}{\zeta_E(s)}=-\frac{\Gamma(s)(2^{s}-1)\cos\left(\frac{\pi s}2\right)}{(2^{s-1}-1)\pi^{s}},
\end{equation}
then from (\ref{Riemann-Euler1}) and (\ref{fe}), he got the functional equation of $\zeta(s).$

Note that 
$$\Gamma(-2n+\epsilon)=\frac{\Gamma(\epsilon+1)}{(-2n+\epsilon)\cdots(-1+\epsilon)}\frac1\epsilon$$
and
$$\zeta_E(-2n)=0.$$
Corollary \ref{cor2} can be used to obtain asymptotic series for $\zeta_E(s)$ of odd positive integer argument.
Indeed, put $s=-2n+\epsilon~(n\in\mathbb N)$ in \eqref{fe} we have
\begin{equation}\label{d-Ezeta}
\begin{aligned}
\zeta_E(2n-\epsilon+1)&=2(2\pi)^{2n-\epsilon}\Gamma(-2n+\epsilon)\frac{1-2^{-2n+\epsilon}}{1-2^{2n-\epsilon+1}}
\zeta_E(-2n+\epsilon) \\
&\quad\times\cos\left(\frac{\pi(-2n+\epsilon)}{2}\right) \\
&=2(2\pi)^{2n-\epsilon}\frac{\Gamma(\epsilon+1)}{(-2n+\epsilon)\cdots(-1+\epsilon)}
\frac{1-2^{-2n+\epsilon}}{1-2^{2n-\epsilon+1}} \\
&\quad\times\left(\frac{\zeta_E(-2n+\epsilon)-\zeta_E(-2n)}\epsilon\right)\cos\left(\frac{\pi(-2n+\epsilon)}{2}\right).
\end{aligned}
\end{equation}
Then taking $\lim_{\epsilon\to0}$ we immediately get
\begin{equation}\label{ft-eq1}
\zeta_E(2n+1)=\frac{2(-1)^n(2\pi)^{2n}}{(2n)!}\frac{2^{-2n}-1}{2^{2n+1}-1}\zeta'_E(-2n).
\end{equation}
Therefore, by using \eqref{ft-eq1} and Corollary \ref{cor2} with $q=1,$ we obatin the following approximation
\begin{equation}\label{ft-eq2}
\begin{aligned}
\zeta_E(2n+1)&\approx
\frac{2(-1)^n(2\pi)^{2n}}{(2n)!}\frac{2^{-2n}-1}{2^{2n+1}-1} \\
&\quad\times\left(\frac14-\frac12\sum_{k=1}^nE_{2k+1}(0)\sum_{j=0}^{2k} \binom{2n}j\frac{(-1)^j}{2k-j+1} \right.  \\
&\qquad\quad\left.-\frac{(2n)!}{2}\sum_{k=n+1}^\infty\frac{E_{2k+1}(0)}{(2k+1)(2k)\cdots(2k-2n+1)} \right).
\end{aligned}
\end{equation}
As examples, we have
\begin{equation}\label{ft-eq3}
\zeta_E(3)\approx
\frac{3\pi^2}{7}\left(\frac14-\frac18\sum_{j=0}^2\binom 2j\frac{(-1)^j}{2-j+1}
-\sum_{k=2}^\infty\frac{E_{2k+1}(0)}{(2k+1)(2k)(2k-1)}\right).
\end{equation}
Here the symbol ``$\approx$" has its usual meaning in this context, namely, that the first $n$ terms of the asymptotic series on the right-hand side of \eqref{ft-eq3} yield an approximation to $\zeta_E(3)$ with an error less than the magnitude of 
the $(n + 1)$th term. Under these conditions, the best approximation to $\zeta_E(3)$ is obtained by keeping five terms in the series,  
\begin{equation}\label{ft-eq4}
\zeta_E(3)\approx
\frac{3\pi^2}{7}\left(\frac14-\frac18\sum_{j=0}^2\binom 2j\frac{(-1)^j}{2-j+1}
-\frac{E_{5}(0)}{5\cdot4\cdot3}\right)=\frac{13\pi^2}{140}.
\end{equation}
This result may be regarded as an analogue of \cite[(23)]{Ru}. In particular, Rudaz \cite{Ru} studied the large-$q$ asymptotic behavior 
of the $s$-derivative of the Hurwitz zeta function $\zeta(s,q)$ evaluated at negative integer values of $s.$

We consider the remainder term $R_N(s,q)$ with $s=0.$ Let $N\in\mathbb N$ and $q$ be a real number.
Using \cite[p. 519, (3.5)]{CP}, we get
\begin{equation}\label{cp-(3.5)}
\log\Gamma(x)=\int_0^\infty\left((x-1)e^t-\frac{e^{-t}-e^{-xt}}{1-e^{-t}}\right)\frac{\d t}{t}, \quad x>0.
\end{equation}
Then by the integral representation (see \cite[p. 519]{CP})
\begin{equation}\label{cp-(3.5)-2}
\log x=\int_0^\infty\frac{e^{-t}-e^{-xt}}{t}{\d t},
\end{equation}
we obtain
\begin{equation}\label{E-eq1}
\begin{aligned}
\log\left(\frac{\Gamma(x+1)}{\Gamma(x+\frac12)}\right)-\frac12\log x
&=\int_0^\infty\left(-\frac{1}{e^{\frac t2}+1}+\frac12 \right)e^{-xt}\frac{\d t}{t} \\
&=-\frac12\int_0^\infty\left(\frac{2}{e^{t}+1}-1 \right)e^{-2xt}\frac{\d t}{t}.
\end{aligned}
\end{equation}
This implies
\begin{equation}\label{E-eq2}
\int_0^\infty\left(\frac{2}{e^{t}+1}-1 \right)e^{-xt}\frac{\d t}{t}=
2\log\left(\frac{\Gamma(\frac{x+1}2)}{\Gamma(\frac x2+1)}\right)+\log \left(\frac x2\right).
\end{equation}
Therefore, by (\ref{E-eq2}) and with $s=0$ in (\ref{rem}), it follows that
\begin{equation}\label{RN(0,q)}
\begin{aligned}
R_{N}(0,q)&=\int_0^\infty\left(\frac{2}{e^t+1}-1-\sum_{k=0}^{N-1}\frac{E_{2k+1}(0)}{(2k+1)!}
t^{2k+1}\right)e^{-qt}\frac{\d t}t \\
&=2\log\left(\frac{\Gamma(\frac{q+1}2)}{\Gamma(\frac q2+1)}\right)+\log \left(\frac q2\right)
-\sum_{k=0}^{N-1}\frac{E_{2k+1}(0)}{(2k+1)!}\frac1{q^{2k+1}}.
\end{aligned}
\end{equation}
From (\ref{rem-te1})  and Corollary \ref{cor1}, we can write
$\zeta_E(0,q)=\frac12-\frac12 R_N(0,q)=\frac12.$ Hence, $R_N(0,q)\sim 0$ as $q\to\infty.$
Using (\ref{RN(0,q)}), the following asymptotic expansion holds
\begin{equation}\label{RN(0,q)-1}
2\log\left(\frac{\Gamma(\frac{q+1}2)\sqrt{\frac q2}}{\Gamma(\frac q2+1)}\right)
\sim\sum_{k=0}^{N-1}\frac{E_{2k+1}(0)}{(2k+1)!}\frac1{q^{2k+1}}
\end{equation}
as $q\to\infty.$ Replacing $q$ and $2q,$ we obtain the asymptotic formula
\begin{equation}\label{ex-quo}
\begin{aligned}
\frac{\Gamma(q+1)}{\Gamma(q+\frac{1}2)}
&\sim\sqrt q\exp\left(-\sum_{k=0}^{N-1}\frac{E_{2k+1}(0)}{2^{2k+2}(2k+1)!}\frac1{q^{2k+1}}\right) \\
&\sim\sqrt q\left(1+\frac{1}{8q}+\frac{1}{128q^2}-\frac{5}{1024q^3}-\frac{21}{32768q^4}+\cdots\right),
\end{aligned}
\end{equation}
if $N$ is sufficiently large (see \cite[p. 356, (1.6)]{BE}).

We close this section by considering one or two particularly simple cases of the $s$-derivative of
$\zeta_E(s,q)$ which can be expressed in terms of the remainder $R_N(s,q).$
Applying Example \ref{ex-n=01}, the asymptotic expansions will be written as
\begin{equation}\label{z(0,q)}
\zeta_E'(0,q)=-\frac12\log q +\frac14q^{-1}-\frac12\sum_{k=1}^{N-1}\frac{E_{2k+1}(0)}{2k+1}\frac1{q^{2k+1}}
+R_N^{(\partial\zeta_E)}(0,q)
\end{equation}
and
\begin{equation}\label{z(-1,q)}
\begin{aligned}
\zeta_E'(-1,q)&=-\frac12q\log q+\left(\frac14+\frac14\log q\right)+\frac12q\sum_{k=1}^{N-1}\frac{E_{2k+1}(0)}{(2k+1)(2k)}
\frac1{q^{2k+1}} \\
&\quad+R_N^{(\partial\zeta_E)}(-1,q)
\end{aligned}
\end{equation}
for $N\in\mathbb N,$ where $R_N^{(\partial\zeta_E)}(0,q)=O_N(|q|^{-2N-1})$
and $R_N^{(\partial\zeta_E)}(-1,q)=O_N(|q|^{-2N})$ as $|q|\to\infty$ in the sector $|\arg q|\leq\pi-\delta,$ 
with $\delta>0$ being fixed (see \cite[p. 9, (4.9)]{Ne}).
Differentiating the right-hand side of (\ref{rem-te1}) with respect to $s$ and comparing the results with (\ref{z(0,q)}) 
and (\ref{z(-1,q)}), we have
$$R_N^{(\partial\zeta_E)}(0,q)=\lim_{s\to0}\frac{-\frac12 R_N(s,q)}{s}
=-\frac12\frac{\partial}{\partial s}R_N(s,q)\biggl|_{s=0}$$
and
$$R_N^{(\partial\zeta_E)}(-1,q)=\lim_{s\to-1}\frac{\frac12 qR_N(s,q)}{s(s+1)}
=-\frac12 q\frac{\partial}{\partial s}R_N(s,q)\biggl|_{s=-1}.$$

We remark here that the remainder terms of the known asymptotic expansions of the
polygamma functions, the gamma function, the Barnes $G$-function and the $s$-derivative of the
Hurwitz zeta function $\zeta(s,q)$ have been evaluated for large-$q$ in \cite[Section 4]{Ne}.

%\section*{Acknowledgement} 
%Min-Soo Kim is supported by the National Research Foundation of Korea (NRF) grant funded by 
%the Korea government (MSIT) (No. NRF-2022R1F1A1065551).  

\bibliography{central}

\end{document}